\documentclass[a4paper,reqno,oneside,11pt]{amsart}
\usepackage{fullpage}

\usepackage{amsmath}
\usepackage{amsfonts}
\usepackage{amssymb}
\usepackage{verbatim}
\usepackage[colorlinks]{hyperref}

\title{Spanning structures and universality in sparse hypergraphs}
\author{Olaf Parczyk}
\author{Yury Person}
\address{Goethe-Universit\"at, Institut f\"ur Mathematik,
  Robert-Mayer-Str. 10, 60325 Frankfurt am Main, Germany}
\email{parczyk | person@math.uni-frankfurt.de}
\date{\today}

\newtheorem{theorem}{Theorem}[section]
\newtheorem{lemma}[theorem]{Lemma}
\newtheorem{proposition}[theorem]{Proposition}
\newtheorem{corollary}[theorem]{Corollary}

\newtheorem{definition}[theorem]{Definition}

\newcommand{\Gnp}{G(n,p)}
\newcommand{\Lrlk}{L^{({r})}(\ell,k)}
\newcommand{\Hnp}{\mathcal{H}^{(r)}(n,p)}

\newcommand{\Hnmp}{\cH(n,p\binom{n}{r})}
\newcommand{\cH}{\mathcal{H}}

\newcommand{\cF}{\mathcal{F}}

\newcommand{\edgeN}{\binom{n}{r}}
\newcommand{\hatp}{\hat{p}}

\newcommand{\PP}{\mathbb{P}}
\newcommand{\EE}{\mathbb{E}}
\newcommand{\Qrd}{Q^{(r)}(d)}
\newcommand{\Var}{\mathrm{Var}}
\newcommand{\Nat}{\mathbb{N}}
\newcommand{\eps}{\varepsilon}
\newcommand{\Srn}{S^r_n}
\newcommand{\Krn}{K^{({r})}_n}
\newcommand{\Crl}{C^{(r,\ell)}}
\newcommand{\SHp}{S_H^\prime}
\newcommand{\FnD}{\mathcal{F}^{(r)}(n,\Delta)}
\newcommand{\link}{\operatorname{link}}
\newcommand{\Img}{\operatorname{Im}}
\newcommand{\F}{\mathcal{F}}
\newcommand{\cK}{\mathcal{K}}
\newcommand{\cL}{\mathcal{L}}

\newcommand{\PnD}{P^{(r)}(\Delta)}

 \thanks{
    This research was supported by DFG grant PE 2299/1-1. 
  }

\begin{document}

\begin{abstract}
 In this paper the problem of finding various spanning structures in random hypergraphs is studied. We notice that  a general result of Riordan [\emph{Spanning subgraphs of random graphs}, Combinatorics, Probability \& Computing \textbf{9} (2000), no.\ 2, 125--148] can be adapted from random graphs  to random $r$-uniform hypergaphs and provide sufficient conditions when a random $r$-uniform hypergraph $\Hnp$ contains a given spanning structure a.a.s. We also discuss several spanning structures such as cube-hypergraphs, lattices, spheres and Hamilton cycles in hypergraphs.
 
 Moreover, we study universality, i.e.\ when does an $r$-uniform hypergraph contain \emph{any} hypergraph on $n$ vertices and with maximum vertex degree bounded by $\Delta$?
 For $\Hnp$ it is shown that  this holds for $p= \omega \left((\ln n/n)^{1/\Delta}\right)$ a.a.s.\ by combining approaches taken by Dellamonica, Kohayakawa, R{\"o}dl and Ruci{\'n}ski
 [\emph{An improved upper bound on the density of universal random graphs}, Random Structures Algorithms \textbf{46} (2015), no.~2, 274--299] and of Ferber, Nenadov and Peter  [\emph{Universality of random graphs and  rainbow embedding}, 
Random  Structures Algorithms, to appear].  Furthermore it is shown that the  random graph $G(n,p)$ for appropriate $p$ and 
 explicit constructions due to Alon, Capalbo, Kohayakawa, R\"odl,  Ruci\'nski and Szemer\'edi~\cite{ACKRRS01} and Alon and Capalbo~\cite{alon2007sparse,alon2008optimal} of universal graphs yield constructions
 of universal hypergraphs that are sparser than the random hypergraph $\Hnp$ with $p= \omega \left((\ln n/n)^{1/\Delta}\right)$.
 
\end{abstract}

\maketitle

\section{Introduction}

Finding spanning subgraphs is a well studied problem in random graph theory, see e.g.\  the following monographs on random graphs~\cite{Bolobas_RandomGraphs,JansonLuczakRucinski_RandomGraphs}. In the case of hypergraphs  less is known and it is natural to study the corresponding spanning problems for random hypergraphs.

An $r$-uniform hypergraph $H$ is a tuple $(V,E)$, where $V$ is its vertex set and  $E \subseteq \binom{V}{r}$ 
the set of  edges in $H$. The random $r$-uniform hypergraph $\Hnp$ is the probability space of all labelled $r$-uniform hypergraphs with the vertex set $[n]:=\{ 1, 2, \ldots ,n \}$ where each edge $e \in \binom{[n]}{r}$ is chosen independently of all the other edges with probability $p$. Thus, for $r=2$ this is a standard model of the random graph $G(n,p)$. 
 Let $H=H^{(i)}$ be a sequence of fixed $r$-uniform hypergraphs with $n$ vertices, where $n=n(i) \rightarrow \infty$.
Then we say that $\Hnp$ contains the hypergraph $H$ asymptotically almost surely (a.a.s.) if the probability that $H^{(i)} \subseteq \Hnp$  tends to $1$ as $n$ tends to infinity. We say that $\hatp$ is a threshold function  if $\PP[H\subseteq \Hnp]\to 0$ for $p\ll\hatp$  and $\PP[H\subseteq \Hnp]\to 1$ for $p\gg\hatp$  as $n$  tends to infinity. 

 It was shown by Bollob\'as and Thomason~\cite{BolTho87}   that all nontrivial monotone properties have a threshold function. Since subgraph containment is a monotone property it is natural to study the threshold functions for appearance of various structures in random graphs and hypergraphs. The purpose of this paper is to obtain generalizations of several results about spanning structures in random graphs to random hypergraphs. In particular we extend a result of Riordan~\cite{Riordan} on containment of a  given \emph{single} spanning graph in $G(n,p)$  and we  generalize a result of Dellamonica, Kohayakawa, R{\"o}dl and Ruci{\'n}ski~\cite{DKRR15} on 
 \emph{universality}  for the class of bounded degree spanning subgraphs in $G(n,p)$ to random hypergraph $\Hnp$. 
Furthermore, we show that the  random graph $G(n,p)$ for appropriate $p$ and 
 explicit constructions  of universal graphs due to Alon, Capalbo, Kohayakawa, R\"odl,  Ruci\'nski and Szemer\'edi~\cite{ACKRRS01} and Alon and Capalbo~\cite{alon2007sparse,alon2008optimal} yield constructions
 of universal hypergraphs that are sparser than the random hypergraph $\Hnp$ with $p= \omega \left((\ln n/n)^{1/\Delta}\right)$.

\subsection{Single spanning structures}
 First spanning structures considered in graphs were perfect matchings~\cite{ErdRen66} and Hamilton cycles~\cite{Kor76,Pos76} (see also~\cite[Chapter~8]{Bolobas_RandomGraphs} and references therein). More recently, the thresholds for the appearance of (bounded degree) spanning trees~\cite{AloKriSud07,Kri10,HefKriSza12,ferber2013universality,JohKriSam13,KahLubWor14b,KahLubWor14a} were studied as well, for the currently best bounds see~\cite{M14a,M14b}. 

Alon and F\"uredi~\cite{AlonFuredi_SpanningSubgraphs} studied the question when the random graph $\Gnp$
 contains a given graph $G$ of bounded maximum degree $\Delta$ hereby proving the bound $p\ge C(\ln n/n)^{1/\Delta}$ for some absolute constant $C>0$.  
  In~\cite{Riordan} Riordan proved quite a general theorem applicable to various graphs in particular determining the threshold functions for the appearance of spanning  hypercubes and lattices. K\"uhn and Osthus~\cite{KO12} determined an approximate threshold for the appearance of a square of a Hamilton cycle. 
  Finding thresholds for factors of graphs and hypergraphs was long an open problem where breakthrough was achieved by Johansson, Kahn and Vu~\cite{JohanssonKahnVu_FactorsInRandomGraphs}. Kahn and Kalai~\cite{KahnKalai_ThresholdsAndExpectationThresholds} have a general conjecture about the thresholds for the appearance of a given structure (which roughly states that the threshold $p$ with $\PP(G\subseteq \Gnp)=1/2$ for containment of $G$ is within a factor of $O(\ln n)$ from $p_E$ at which the expected number of copies of $G$ in $G(n,p_E)$ is $1$, where $p_E$ is the so-called expectation threshold).
  
When one turns to hypergraphs, so apart from perfect matchings and general factors~\cite{JohanssonKahnVu_FactorsInRandomGraphs}, the only other spanning structures that were studied more recently are \emph{Hamilton cycles}. An $\ell$-overlapping Hamilton cycle is an  $r$-uniform hypergraph  such that for some cyclic ordering of $[n]$ 
and an ordering of the edges, every edge $e_i$ consists of $r$ consecutive vertices and for any two consecutive edges $e_i$ and $e_{i+1}$ it holds $|e_i \cap e_{i+1}|=\ell$ (this requires that $r-\ell$ divides $n$ and thus such $\ell$-overlapping Hamilton cycle has $n/(r-\ell)$ edges).  We say that a hypergraph is $\ell$-hamiltonian if it contains an $\ell$-overlapping Hamilton cycle. Frieze~\cite{Fri10}  determined 
the threshold for the appearance of $1$-overlapping $3$-uniform Hamilton cycles to be $\Theta(\ln n/n^2)$ (when $4|n$) and
Dudek and Frieze~\cite{DudekFrieze_LooseHamiltonCycles} extended the result to higher uniformities ($2(r-1)|n$). 
 The divisibility requirement was improved to the optimal one ($(r-1)|n$) by  Dudek, Frieze, Loh and Speiss~\cite{DFLS_opt}, see also Ferber~\cite{Fer15}.   Subsequently, Dudek and Frieze~\cite{Dudek-Frieze_hamiltonian} determined thresholds for general $\ell$-overlapping Hamilton cycles and a randomized algorithm to find $(r-1)$-overlapping Hamilton cycles was given in~\cite{ABHKPonline}. For a table of the known thresholds we refer the reader to~\cite{Dudek-Frieze_hamiltonian}, but generally $\omega(n^{\ell-r})$ is an asymptotically optimal threshold for $\ell$-Hamiltonicity (for $\ell\ge 2$ and in most situations even more precise results are known). Here and in the following  $\omega(f)$ denotes any function $g$ such that $g(n)/f(n)\rightarrow \infty$ as $n\to \infty$. For another type of Hamilton cycles,  the so-called \emph{Berge} Hamilton cycles, the threshold was determined recently by Poole~\cite{Poole14}.

 In the first part of this paper we observe that Riordan's result~\cite{Riordan} can be extended to $r$-uniform hypergaphs leading to a general theorem about spanning structures in random hypergraphs. We will recover results of Dudek and Frieze~\cite{Dudek-Frieze_hamiltonian} on $\ell$-hamiltonicity ($2\le\ell<r$) and also discuss thresholds for other spanning structures such as hypercubes, lattices, spheres and powers of Hamilton cycles in hypergraphs. 

Let $H=(V,E)$ be an $r$-uniform hypergraph with $n$ vertices. We write $v(H)$ for $|V|$ and $e(H)$ for $|E|$. We denote by $\deg(v)$  the degree of a vertex $v$ in $H$: $\deg(v):=|\{e\colon v\in e\}|$, and  $\Delta(H)$ is defined to be  the maximum vertex degree in $H$, i.e.\ $\Delta(H):=\max_{v\in V}\deg(v)$. Let $e_H(v)=\max \{ e(F) : F \subseteq H, v(F)=v \}$, then 
the following parameter introduced in~\cite{Riordan} will be responsible for the upper bound on the threshold
\begin{align*}
\gamma(H) = \max_{r+1 \le v \le n} \left\{ \frac{e_H(v)}{v-2} \right\}.
\end{align*}
Our first result is the following. 
\begin{theorem}\label{theorem_main}
Let $r\ge 2$ be an integer and  $H=H^{(i)}$ be a  sequence of $r$-uniform hypergraphs with $n=n(i)$ vertices, $\Delta=\Delta(H)$ and 
 $e(H)=\alpha \binom{n}{r} = \alpha (n)\binom{n}{r}$ edges. Let  $p = p(n)\colon \Nat\to [0,1)$. 
  If the following conditions are satisfied 
\begin{align}
\label{eq_theocond}
 \alpha \binom{n}{r} > \frac{n}{r}\quad\text{and}\quad  p\binom{n}{r}\rightarrow \infty,
\end{align}
and
\begin{align}
\label{eq_npinfty}
n p^{\gamma(H)} \Delta^{-4} \rightarrow \infty,
\end{align}
then a.a.s.\ the random $r$-uniform hypergraph $\Hnp$ contains a copy of $H$.
\end{theorem}
 We remark, that for $r=2$ this was already shown by Riordan in~\cite[Theorem~2.1]{Riordan} using the second moment argument. 
 In fact, the proof for hypergraphs will follow allong the lines of his original argument, but requires adaptations at various places. We provide the details below in Section~\ref{sec_proof} and in Section~\ref{sec_applications} we discuss its applications to some particular spanning hypergraphs. 

\subsection{Universality}

For a family $\cF$ of $r$-uniform hypergaphs we say that an $r$-uniform hypergraph $H$ is $\mathcal{F}$-universal if every hypergraph $F \in \mathcal{F}$ occurs as a copy in $H$. Let $\FnD$ denote the family of all $r$-uniform hypergraphs $F$ with maximum degree $\Delta(F) \le \Delta$ on $n$ vertices. 

In the graph case, the first result concerned almost spanning universality  due to Alon, Capalbo, Kohayakawa, R\"odl,  Ruci\'nski and Szemer\'edi~\cite{ACKRRS00} who showed that 
for $p=C (\ln n/n)^{1/\Delta}$ the random graph  $G(n,p)$ is $\cF^{(2)}((1-\eps)n,\Delta)$-universal a.a.s. 
 Then, Dellamonica, Kohayakawa, R{\"o}dl and Ruci{\'n}ski~\cite{DKRR15} proved for any given  $\Delta \ge 3$ that $G(n,p)$ is $\mathcal{F}^{(2)}(n,\Delta)$-universal a.a.s.\ 
provided that $p \ge C \left( \ln n / n \right)^{1/\Delta}$, where $C>0$ is some absolute constant.
Later, Kim and Lee~\cite{kim2014universality} dealt with the missing case $\Delta=2$.
Bringing the maximum density $d(H)=\max_{H^\prime \subseteq H} \left\{ 2 e(H) / v(H) \right\}$ of a graph into the statement, Ferber, Nenadov and Peter~\cite{ferber2013universality} showed that for the universality for all graphs with maximum degree $\Delta$ and maximum density $d$ the probability $p=\omega(\Delta^{12} n^{-1/(2d)} \ln^3 n)$ suffices. Thus, for sparser graphs with  $d < \Delta/2$ this improved~\cite{DKRR15} and an additional advantage is that $\Delta$ is allowed to be a function  of $n$.
Very recently, Conlon, Ferber, Nenadov and \v{S}kori\'{c}~\cite{conlon2015almost}, building on~\cite{ferber2013universality} proved that for every $\eps>0$, $\Delta\ge 3$ and 
$p = \omega(n^{-1/(\Delta-1)} \ln^5 n)$  the random graph $\Gnp$ is $\mathcal{F}^{(2)}((1-\eps)n,\Delta)$-universal a.a.s.\ improving the almost spanning result from~\cite{ACKRRS00}.

Our second result is on the universality of $\Hnp$ for the family  $\FnD$, where we show that a natural bound on $p\ge C (\ln n /n)^{1/\Delta}$ suffices. 

\begin{theorem}\label{theom_universality}
For every $r \ge 2$ and any integer $\Delta \ge 1$, there exists a constant $C>0$, such that for $ p=C \left( \ln n / n \right)^{1/\Delta}$ the random $r$-uniform hypergraph $\Hnp$ is $\FnD$-universal a.a.s.
\end{theorem}

In the proof of Theorem~\ref{theom_universality} (see Section~\ref{sec:proof_universality}) we employ a strategy of Dellamonica, Kohayakawa, R{\"o}dl and Ruci{\'n}ski~\cite{DKRR15}, but a shortcut will be obtained by using similar notions of good properties that were introduced by Ferber, Nenadov and Peter~\cite{ferber2013universality}. 

Another natural problem concerns the \emph{existence} and \emph{explicit constructions} of graphs that are universal for some family of graphs. 
For an excellent survey on this problem see Alon~\cite{Alon10} and the references therein. 
First nearly optimal universal graphs for $\cF^{(2)}(n,\Delta)$ ($\Delta\ge 3$) with $O(n)$ vertices and $O(n^{2-2/\Delta}\ln^{1+8/\Delta}n)$ edges were given by Alon, Capalbo, Kohayakawa, R\"odl,  Ruci\'nski and Szemer\'edi~\cite{ACKRRS01}. 
It was also noted by the same authors in~\cite{ACKRRS00} that any such universal graph has to contain $\Omega(n^{2-2/\Delta})$ edges. 
As mentioned in \cite{alon2007sparse} in the case $\Delta=2$ the square of a Hamilton cycle is $\cF^{(2)}(n,\Delta)$-universal (and thus $2n$ edges are enough in this case).  
Subsequently, Alon and Capalbo gave better constructions of $\cF^{(2)}(n,\Delta)$-universal graphs ($\Delta\ge 3$) with $O(n)$ vertices and $O(n^{2-2/\Delta})$ edges~\cite{alon2007sparse} and another one with $n$ vertices and $O(n^{2-2/\Delta} \ln^{4/\Delta} n)$ edges~\cite{alon2008optimal}.
 We exploit their constructions to obtain sparse $\FnD$-universal hypergraphs.

\begin{proposition}
\label{prop:universality}
For every $r \ge 3$, any integer $\Delta \ge 2$ and  any $\cF^{(2)}(n,\Delta)$-universal graph $G$ there is an $r$-uniform hypergraph $H$ with $V(H)=V(G)$ and $|E(H)|\le |E(G)| n^{r-2}$, which is $\FnD$-universal.
\end{proposition}
 
It will follow from Proposition~\ref{prop:universality} that 
 if $G$ is explicitly constructible then so is the hypergraph $H$. 
This proposition together with the results of Alon and Capalbo~\cite{alon2007sparse,alon2008optimal} yields for  $\Delta\ge 2$  constructions of $\FnD$-universal hypegraphs with $O(n)$ vertices and 
$O(n^{r-2/\Delta})$ edges and with $n$ vertices and $O(n^{r-2/\Delta} \ln^{4/\Delta}n)$ edges. Note that the case $\Delta=1$ is trivial. 

Our next proposition shows that, for $r\ge 4$ there exist even sparser $\FnD$-universal hypergraphs. These are obtained from appropriate universal random graphs from~\cite{conlon2015almost,DKRR15}.
\begin{proposition}\label{prop:shadow_univ}
 Let $r\ge 3$ and $\Delta\ge 2$ be integers. Then there exists an $\FnD$-universal hypergraph $H_1$ with $n$ vertices and $\Theta\left(n^{r-\frac{r}{2\Delta}}(\ln n)^{\frac{r}{2\Delta}}\right)$ edges. 
Moreover, for every $\eps>0$ there exists an $\FnD$-universal hypergraph $H_2$ with $(1+\eps)n$ vertices and 
$\omega\left(n^{r-\frac{\binom{r}{2}}{(r-1)\Delta-1}}(\ln n)^{5\binom{r}{2}}\right)$ edges.
\end{proposition}

 These hypergraphs are thus sparser than the random universal hypergraph from 
Theorem~\ref{theom_universality}. On the other hand, any $\FnD$-universal hypergraph has to contain  $\Omega(n^{r-r/\Delta})$ edges and thus the exponent in its density is off by at most  the factor of $2$. 
We discuss Propositions~\ref{prop:universality} and~\ref{prop:shadow_univ} in Section~\ref{sec:construction}.
 
\subsection{Notation}

Let $H=(V,E)$ be an $r$-uniform hypergraph. The hypergraph induced by a subset of the vertices $W \subseteq V$ in $H$ is denoted by $H[W]:=\left(W,E(H)\cap\binom{W}{r}\right)$. The shadow graph $H^\prime$ is obtained from $H$ by replacing every edge $e\in E(H)$ by all possible $\binom{r}{2}$ subsets of cardinality two (we ignore multiple edges). By $\Krn$ we denote the complete $r$-uniform hypergraph $\left([n],\binom{n}{r}\right)$. 

An alternating sequence of vertices and edges $v_1,e_1,v_2,e_2,\dots,v_{t},e_t,v_{t+1}$ is called a path of length $t$ from $v_1$ to $v_{t+1}$ if $v_i$, $v_{i+1}\in e_{i}$ for all $i\in[t]$. 
 If there is a path from $u$ to $v$, then we say that $u$ and $v$ are connected. This defines an equivalence relation on $V$. Similarly to the graph case, we say that a hypergraph $H$ is connected if there is a path between any two vertices of $H$. A component in an $r$-uniform hypergraph is a maximally connected subhypergraph. 
The distance between two vertices $u$ and $v$ in $H$ ist the minimal length over all paths from $u$ to $v$, and if they are in different components then we set it to $\infty$.

The neighbourhood $N_{H}(v)$ of a vertex $v$ is the set of vertices which are contained in an edge together with $v$
\begin{align*}
 N_{H}(v)=\{ w \in V \setminus \{ v \} : \exists \, e \in E \text{ s.t.\ } \{ w, v \} \subseteq e \}.
\end{align*}
For a subset of the vertices $W \subseteq V$, the neighbourhood in $H$ is $N_H (W) = \bigcup_{w \in W} N_H(w)$.
The set $W$  is called $t$-independent in a hypergraph $H$, if the distance between $v\in W$ and $w \in W$ in $H$  is at least $t+1$.
A $1$-independent set is independent in the usual sense.

To simplify readability, we will omit in the calculations  floor and ceiling signs whenever they  are not crucial in the arguments.

\section{Proof of Theorem~\ref{theorem_main}}\label{sec_proof}

The overall proof strategy of Theorem~\ref{theorem_main} is the same as Riordan's in~\cite{Riordan}, which is 
an elegant second moment argument. In fact, a large part of the proof proceeds along the same lines and  we are thus  going to use the same notation. We will also provide reference to the relevant part of~\cite{Riordan}, especially  for lemmas that hold verbatim or by a  straightforward modification for hypergraphs, but we also include some of these for the sake of readability. Still, some of the steps require more effort to be generalized (in particular Lemma~\ref{le_SH'} below) and we provide full details for them. We try to be brief anyway.

The proof deals instead of $\Hnp$ with the related model $\Hnmp$, which is the probability space of all labelled $r$-uniform hypergraphs with the vertex set $[n]$ and exactly $p\binom{n}{r}$ edges with a uniform measure. Thus, for $r=2$ this is  the standard model $G(n,m)$. A corresponding statement in the model $\Hnp$ is  then obtained from $\Hnmp$ by a standard argument, i.e.\ conditioning on the number of edges in $\Hnp$.

One considers the random variable $X$ which counts  copies of $H$ in  $\Hnmp$ and  analyzes the quantity $f:=\EE(X^2)/(\EE X)^2$. It is enough to show that $f=1+o(1)$, since then one infers by Chebyshev's inequality:
\begin{equation}\label{eq:chebyshev}
 \PP[X=0]\le \PP[\left|X-\EE X\right|\ge\EE X]\le \frac{\Var(X)}{\EE(X)^2}=f-1=o(1).
\end{equation}

Before we give more details, let us briefly state the steps that are geared towards the estimation of $f$ as done in~\cite{Riordan}, since this is the path we are going to pursue as well:
\begin{enumerate}
\item pondering $f$, it is shown that $f\le (1+o(1))e^{- \frac{1-p}{p} \alpha^2 \edgeN} S_H$, 
where $S_H$ is a sum that depends on all subhypergraphs of $H$, which will be introduced below;
\item then it is shown that $S_H \le (1+o(1)) e^{\frac{1-p}{p} \alpha^2 \edgeN} S^\prime_H$, 
where $S^\prime_H$ runs only over certain good subhypergraphs of $H$ -- 
 this step requires most adaptation and we provide full details in Lemma~\ref{le_SH'} below;
\item one can further simplify $S^\prime_H$ and bound it above by another quantity $T^\prime_H$ -- 
this is done in  Lemma~\ref{le_XFdelta} and its proof is sketched in the Appendix;
\item in the penultimate step, $T^\prime_H$ is bounded by $e^{T_H^{\prime \prime}}$, where $T_H^{\prime \prime}$ is the sum over all good connected hypergraphs;
\item finally, it is shown $T_H^{\prime \prime}=o(1)$ (we sketch a proof in Lemma~\ref{le_TH´´} that can be found in the Appendix) and combining the estimates, the desired bound on $f$ follows:

\centering{
\begin{multline*}
f\le (1+o(1))e^{- \frac{1-p}{p} \alpha^2 \edgeN} S_H\le 
(1+o(1)) e^{- \frac{1-p}{p} \alpha^2 \edgeN} e^{\frac{1-p}{p} \alpha^2 \edgeN} S^\prime_H\le\\
(1+o(1))T_H^{\prime}\le (1+o(1))e^{T_H^{\prime \prime}}\le (1+o(1)) e^{o(1)}=1+o(1).
\end{multline*}}
\end{enumerate}

Before we proceed, let us collect some useful estimates that involve $\alpha$ and $p$ for future reference.
\begin{lemma}\label{lem:estimates}
Suppose that the conditions~\eqref{eq_theocond} and~\eqref{eq_npinfty} hold. Then, we have
\begin{align*}
&n^{\frac{3-2r}{2}} p^{-1} \Delta^{4r-6}\rightarrow 0\quad\text{and}\quad \Delta=o(n^{1/4}),\\
 &\alpha^3 \edgeN p^{-2}=o(1),\\
&\alpha p^{-1}\Delta=o(n^{-1/2}), \,\,
 p^{-1}\Delta^2 n^{2-r}=o(n^{1/2})\,\,\, \text{and}\quad 
 p^{-1}\alpha^2\binom{n}{r}=o(n^{1/2}). 
 \end{align*}
\end{lemma}
\begin{proof}
Observe first that $\alpha\binom{n}{r}>n/r$ implies $\gamma(H)\ge \frac{2}{2r-3}$. Since $p\le 1$, it follows from~\eqref{eq_npinfty} that $\Delta=o(n^{1/4})$, and rearranging yields with $\gamma(H)\ge \frac{2}{2r-3}$ that  $n^{\frac{3-2r}{2}} p^{-1} \Delta^{4r-6}\rightarrow 0$. 

Now we notice immediately that  $p=\omega\left((\Delta^4/n)^{\frac{2r-3}{2}}\right)$.
 Since $\alpha \edgeN \le \Delta n/r$ it follows $\alpha\le \Delta \binom{n-1}{r-1}^{-1}$. The combination of the two estimates yields $\alpha^3 \edgeN p^{-2}=o(1)$.
 
To obtain the remaining estimates, we combine the lower bound on $p=\omega\left((\Delta^4/n)^{\frac{2r-3}{2}}\right)$ with  $\alpha\le \Delta \binom{n-1}{r-1}^{-1}$, thus obtaining 
$
\alpha p^{-1}\Delta\le p^{-1}\Delta^2\binom{n-1}{r-1}^{-1}=o(n^{-1/2}) 
$
and 
$
p^{-1}\alpha^2\binom{n}{r}=o(n^{1/2}).
$
\end{proof}

The proof starts by writing $f$ as the sum over all $2^{e(H)}$ subhypergraphs $F$ of $H$ that  involve $X_F(H)$ and $X_F(\Krn)$, which are the number of subhypergraphs of $H$ (resp.\ of $\Krn$) that are isomorphic to $F$.
 The following lemma is  a consequence of Lemmas~3.1,~3.2 and~4.1 from~\cite{Riordan} (these involve only binomial coefficients and thus can be applied verbatim to hypergraphs). 
\begin{lemma}
\label{lem_combi}
 Suppose that $\alpha \edgeN , p\edgeN \rightarrow \infty$ and that 
 $\alpha^3 \edgeN p^{-2} \rightarrow 0$ (as $n$ tends to infinity). Then with $c = \frac{1-p}{p-2 \alpha}$ we get
\begin{align}
\label{eq_fboundc}
 f \le (1+o(1)) e^{-\frac{1-p}{p} \alpha^2 \edgeN} \sum_{F \subseteq H} c^{e(F)} \frac{X_F(H)}{X_F(\Krn)}. 
\end{align}
\end{lemma}
The  third condition above ($\alpha^3 \edgeN p^{-2} \rightarrow 0$) holds by Lemma~\ref{lem:estimates}. 

Notice  that every component in an $r$-uniform hypergraph has either one vertex (isolated vertex) or at least $r$ vertices. 
We define the function $r(F):=n-k_1(F)-(r-1)k_r(F)$ where $n$ is the number of vertices in $F$, $k_1(F)$ is the number of isolated vertices in $F$ and $k_r(F)$ is the number of components in $F$, which are not isolated vertices. As in~\cite{Riordan} one introduces the sum
\begin{align}
\label{eq_defSH}
S_H= \sum_{F \subseteq H} (p^{-1}-1)^{e(F)} (1+n^{-\frac{1}{2}})^{r(F)} \frac{X_F(H)}{X_F(\Krn)}.
\end{align}
This allows us to get the bound of the form (this is Lemma~4.2 in~\cite{Riordan})
\[
f\le (1+o(1))e^{- \frac{1-p}{p} \alpha^2 \edgeN} S_H,
\]
 which follows by estimating $\ln\left(\frac{p}{p-2\alpha}\right)\le \frac{3\alpha}{p}$, $e(F)\le \Delta r(F)$ and the use of 
 $ \alpha p^{-1}\Delta=o(n^{-1/2})$ (which follows from Lemma~\ref{lem:estimates}).

Next one would like to estimate $S_H$ by the following sum
\begin{align}
\label{eq_defSH'}
S^\prime_H= \sum_{F \subseteq H}^\prime (p^{-1}-1)^{e(F)} 2^{r(F)} \frac{X_F(H)}{X_F(K_n^{(r)})},
\end{align}
where $\sum^\prime$ is the sum over subhypergraphs $F\subseteq H$ such that none of the components of $F$  consists of a single (isolated) edge (such hypergraphs are referred to as \emph{good} in~\cite{Riordan}). 

The following lemma has the same conclusion as Lemma~4.3 from~\cite{Riordan}. In the case of $r$-uniform hypergraphs ($r\ge 3$) one needs to be more careful and the estimates are somewhat different from those in~\cite{Riordan}.  Therefore we provide its full proof below.
\begin{lemma}
\label{le_SH'}
If $H$ is any $r$-uniform hypergraph with maximum degree $\Delta$ and the conditions~\eqref{eq_theocond} and~\eqref{eq_npinfty} hold, then  
\begin{align*}
S_H \le (1+o(1)) e^{\frac{1-p}{p} \alpha^2 \edgeN} S^\prime_H.
\end{align*}
\end{lemma}
\begin{proof}
Let $F$ be some hypergraph from the sum $\sum^\prime$ in $S'_H$ (such $F$  we call good). Thus, $F$ is  an $r$-uniform  hypergraph with $v$ isolated vertices and no isolated edges.  We define $S^\prime [F]$ to be the contribution to $\SHp$ that comes from the isomorphism class of a good hypergraph $F\subseteq H$, i.e.\ $S^\prime [F]=(p^{-1}-1)^{e(F)} 2^{r(F)} \frac{X_F(H)^2}{X_F(\Krn)}$. We write $F_t$ for a hypergraph obtained from a good $F$  with $v$ isolated vertices by adding $t\le v/r$ isolated edges to it. Let $S[F]$ be the contribution of all subhypergraphs of $H$ that are isomorphic to $F_i$  for some $i$, where $0\le i\le v/r$. Thus, $S[F]=\sum_{i=0}^{v/r}(p^{-1}-1)^{e(F_i)} (1+n^{-\frac{1}{2}})^{r(F_i)} \frac{X_{F_i}(H)^2}{X_{F_i}(\Krn)}$. 
Every hypergraph from the sum in $S_H$ can be reduced to a good $F$ by deleting all  isolated edges. Therefore we have  $\SHp=\sum S^\prime[F]$ and $S_H=\sum S[F]$, where the sums are over all isomorphism classes of good subhypergraphs of $H$.  To prove the lemma it is thus sufficient to bound $S[F]/S^\prime[F]$ for every good $F\subseteq H$ by $(1+o(1)) e^{\frac{1-p}{p} \alpha^2 \edgeN}$.

Let $F\subseteq H$ be a good hypergraph with $v$ isolated vertices, then 
\begin{align*}
 X_{F_t}(\Krn) = X_F(\Krn)\cdot \frac{1}{t!} \binom{v}{r} \binom{v-r}{r}\cdot \ldots\cdot \binom{v-rt+r} {r}
\end{align*}
and
\begin{align*}
 X_{F_t}(H) \le X_F(H)\cdot \frac{1}{t!} e_H(v) e_H(v-r)\cdot \ldots\cdot e_H(v-rt+r).
\end{align*}
Setting $\beta_w = e_H(w)^2 / \binom{w}{r}$ we obtain
\begin{align*}
 \frac{X_{F_t}(H)^2}{X_{F_t}(\Krn)} \le \frac{X_F(H)^2}{X_F(\Krn)} \frac{\beta_v \beta_{v-r} \dots \beta_{v-rt+r}}{t!}.
\end{align*}
Since $e(F_t)=e(F)+t$ and $r(F_t)=r(F)+t$ we have
\begin{align}
\label{eq_S/S'}
 \frac{S[F]}{S^\prime[F]} \le 2^{-r(F)} (1+n^{-\frac{1}{2}})^{r(F)}\sum_{t=0}^{v/r} (p^{-1}-1)^t(1+n^{-\frac{1}{2}})^t \frac{\beta_v \beta_{v-r} \ldots \beta_{v-rt+r}}{t!}.
\end{align}

Next we take a closer look at the $\beta_w$ terms. 
Since $\Delta(H)\le \Delta$ we can bound $e_H(w)\le w\Delta/r$ and $\beta_w\le \left( \frac{w \Delta}{r} \right)^2 \binom{w}{r}^{-1}\le \frac{\Delta^2 r^{r-2}}{w^{r-2}}$. Therefore we estimate the product of $\beta_w$s as follows 
\begin{align*}
\prod_{i=0}^{t-1} \beta_{v-ir} \le \left( \Delta^2 r^{r-2} \right)^t \left( \prod_{i=0}^{t-1} (v-ir) \right)^{-(r-2)} \le 
\Delta^{2t} \left( \frac{(\lfloor v/r\rfloor-t)!}{\lfloor v/r\rfloor!} \right)^{r-2}.
\end{align*}
By approximating factorials with Stirling's formula we obtain
\[
\prod_{i=0}^{t-1} \beta_{v-ir}\le 
 \left( \frac{e^{r-2}\Delta^{2}}{\lfloor v/r\rfloor^{r-2}} \right)^{t}.
\]
Thus, we further upper bound $S[F]/S^\prime[F]$  by
\begin{equation}\label{eq:sfprime}
 \frac{S[F]}{S^\prime[F]} \le 2^{-r(F)} (1+n^{-\frac{1}{2}})^{r(F)}
 \sum_{t=0}^{v/r} (p^{-1}-1)^t(1+n^{-\frac{1}{2}})^t  \left( \frac{e^{r-2}\Delta^{2}}{\lfloor v/r\rfloor^{r-2}} \right)^{t} \frac{1}{t!}. 
\end{equation}

From Lemma~\ref{lem:estimates} it follows that $p^{-1}\Delta^2 n^{2-r}=o(\sqrt{n})$ and 
therefore 
\begin{equation}\label{eq:nvest}
p^{-1}\Delta^2 v^{2-r}=o((n/v)^{r-2}\sqrt{n})
\end{equation}
 and in the following we will distinguish four cases. 

Suppose $0\le v\le n/(100r \ln n)$. Then we use~\eqref{eq:nvest} to upper bound each term in the sum from~\eqref{eq:sfprime} by $n^v=\exp(n/(100r))$. On the other hand we have $r(F)\ge \frac{n-v}{r}>n/(2r)$. It follows that $2^{-r(F)}$ dominates each of the at most $n/r$ terms in the sum and the factor $(1+n^{-\frac{1}{2}})^{r(F)}$ as well. This gives us $S[F]/S^\prime[F]=o(1)$. If $v=0$ then we trivially have $S[F]/S^\prime[F]=o(1)$ as well.

Next we assume that $n/(100r \ln n)<v\le n - (\ln n)^{r-2}\sqrt{n}$. We can interpret  the sum in~\eqref{eq:sfprime} as the first $v/r+1$ terms in the expansion of $\exp\left( (p^{-1}-1) (1+n^{-1/2})  \frac{e^{r-2}\Delta^{2} }{\lfloor v/r\rfloor^{r-2}} \right)$, which leads to
\begin{equation}\label{eq:crucial_sum}
 \frac{S[F]}{S^\prime[F]}\le 2^{-r(F)} (1+n^{-\frac{1}{2}})^{r(F)} 
\exp\left(2p^{-1}\frac{e^{r-2}\Delta^{2}}{\lfloor v/r\rfloor^{r-2}}\right).
\end{equation}
Again we have $r(F)\ge \frac{n-v}{r}\ge \frac{(\ln n)^{r-2}\sqrt{n}}{r}$, whereas 
$\exp\left(2p^{-1}\frac{e^{r-2}\Delta^{2}}{\lfloor v/r\rfloor^{r-2}} \right)=\exp\left(o((\ln n)^{r-2}\sqrt{n})\right)$ by~\eqref{eq:nvest}. Thus, we have $S[F]/S^\prime[F]=o(1)$.

Assume now that $n - (\ln n)^{r-2}\sqrt{n}< v\le n-\sqrt{n}$. Similarly as in the previous case one gets 
$r(F)\ge \sqrt{n}/r$ and $\exp\left(2p^{-1}\frac{e^{r-2}\Delta^{2}}{\lfloor v/r\rfloor^{r-2}} \right)=\exp\left(o(\sqrt{n})\right)$. Again one gets $S[F]/S^\prime[F]=o(1)$ as before.

Finally, let $v>n-\sqrt{n}$ and we are going to use the inequality~\eqref{eq_S/S'} to estimate $S[F]/S^\prime[F]$. We bound $\beta_w$ with $e(H)^2\binom{w}{r}^{-1}=\alpha^2 \edgeN^2  \binom{w}{r}^{-1}$ which is $\alpha^2\edgeN (1+O(n^{-1/2}))$ for $w\ge n - (r+1)\sqrt{n}$. This gives us
\[
 \sum_{t=0}^{\sqrt{n}} (p^{-1}-1)^t(1+n^{-\frac{1}{2}})^t \frac{\left(\alpha^2\edgeN (1+O(n^{-1/2}))\right)^t}{t!}\le \exp\left(\frac{1-p}{p}\alpha^2\edgeN \left(1+O(n^{-1/2})\right)\right).
\]
By Lemma~\ref{lem:estimates} we have $\frac{1-p}{p}\alpha^2\edgeN n^{-1/2}=o(1)$. Thus, 
\[
\exp\left(\frac{1-p}{p}\alpha^2\edgeN \left(1+O(n^{-1/2})\right)\right) \le (1+o(1))\exp\left(\frac{1-p}{p}\alpha^2\edgeN \right).
\]

As for $t> \sqrt{n}$, we estimate the rest by~\eqref{eq:sfprime} and using~\eqref{eq:nvest} it follows:
\begin{align*}
 \sum_{t=\sqrt{n}}^{v/r} (p^{-1}-1)^t(1+n^{-\frac{1}{2}})^t \left( \frac{e^ {r-2} \Delta^2}{\lfloor v/r \rfloor^ {r-2}} \right)^t \frac{1}{t!}\le
\sum_{t=\sqrt{n}}^{v/r} {o(1)^t} =o(1).
\end{align*}
Combining together we obtain: $\frac{S[F]}{S^\prime[F]} \le (1+o(1)) e^{\alpha^2 \edgeN \frac{1-p}{p}}+o(1)=(1+o(1)) e^{\alpha^2 \edgeN \frac{1-p}{p}}$.

 Therefore, for every good $F$, we get in any of the four possible cases that
\begin{align*}
\frac{S[F]}{S^\prime[F]} \le (1+o(1)) e^{\alpha^2 \edgeN \frac{1-p}{p}}.
\end{align*}
This yields $S_H/\SHp\le (1+o(1)) e^{\alpha^2 \edgeN \frac{1-p}{p}}$ completing the proof.
\end{proof}

So far we have $f\le (1+o(1))e^{- \frac{1-p}{p} \alpha^2 \edgeN} S_H$ and 
$S_H \le (1+o(1)) e^{\frac{1-p}{p} \alpha^2 \edgeN} S^\prime_H$, thus $f\le (1+o(1))\SHp$. In the following 
one bounds $\SHp$ by bounding first $X_F(H)/X_F(\Krn)$.
\begin{lemma}[an adaptation of Lemma~4.4 from~\cite{Riordan}]
\label{le_XFdelta}
 Let $H$ be any $r$-uniform hypergraph with maximum degree $\Delta$ and $F \subseteq H$, then
\begin{align*}
 \frac{X_F(H)}{X_F(\Krn)} \le \frac{\left(2 (r-1)! \Delta\right)^{r(F)}e^{r(F)+(r-2)k_r(F)}}{n^{r(F)+(r-2)k_r(F)}} 
\end{align*}
\end{lemma}
 
 The lemma above bounds $\SHp$ as follows: 
 \[
 \SHp\le \sum^\prime_{F\subseteq H} (p^{-1}-1)^{e(F)}  \frac{\left(4 (r-1)! \Delta\right)^{r(F)}e^{r(F)+(r-2)k_r(F)}}{n^{r(F)+(r-2)k_r(F)}}=:T^\prime_H.\]  Proceeding exactly as in~\cite{Riordan}, one introduces 
 $\psi(F):=(p^{-1}-1)^{e(F)}\frac{\left(4 (r-1)! \Delta\right)^{r(F)}e^{r(F)+(r-2)k_r(F)}}{n^{r(F)+(r-2)k_r(F)}}$ which is multiplicative: it holds $\psi(F_1\cup F_2)=\psi(F_1)\psi(F_2)$ for any two vertex-disjoint hypergraphs $F_1$ and $F_2$ (i.e.\ for all $e\in E(F_1)$ and $f\in E(F_2)$ one has $e\cap f=\emptyset$). Since every good hypergraph consists of maximally connected edge-disjoint good hypergraphs we get
 \begin{align*}
 T^\prime_H=\sum_{F \subseteq H}^{\prime} \psi(F) \le 1 + \sum_{i=1}^\infty \frac{1}{t!} \left( \sum_{F \subseteq H}^{\prime\prime} \psi(F) \right)^t,
\end{align*}
where $\sum^{\prime \prime}$ is the sum over connected good hypergraphs.
 We set $T_H^{\prime \prime} = \sum_{F \subseteq H}^{\prime \prime} \psi(F)$, thus the above shows $T_H^\prime \le e^{T_H^{\prime \prime}}$.

\begin{lemma}[an adaptation of Lemma~4.5~\cite{Riordan}]
\label{le_TH´´}
 For every $r$-uniform hypergraph $H$ on $[n]$ we have
\begin{align}
\label{eq_TH''}
 T_H^{\prime \prime} \le n e^{2r} \sum_{s=r+1}^n \left(\frac{12 r!^2 \Delta^2}{n}\right)^{s-1}  p^{-e_H(s)},
\end{align}
where $\Delta$ is the maximum degree of $H$.
\end{lemma}

Now we are in a position to finish the argument. We further estimate $T_H^{\prime \prime}$ using the above as follows:
\[
 T_H^{\prime \prime} \le n e^{2r} \sum_{s=r+1}^n \left(\frac{12 r!^2 \Delta^2}{n}\right)^{s-1}  p^{-e_H(s)}\le 12 e^{2r} r!^2  \sum_{s=r+1}^n \left(12 r!^2 \Delta^4 p^{-e_H(s)/(s-2)}n^{-1} \right)^{s-2}. 
\]
Therefore we get $ T_H^{\prime \prime}\le 12 e^{2r} r!^2  \sum_{s=r+1}^n \left(12 r!^2 \Delta^4 p^{-\gamma(H)}n^{-1} \right)^{s-2}$, which by condition~\eqref{eq_npinfty} tends to zero as $n$ goes to infinity. Thus, $ T_H^{\prime \prime}=o(1)$, and with $T_H^\prime \le e^{T_H^{\prime \prime}}$ and $f\le (1+o(1))\SHp\le T_H^\prime$ we obtain $f\le 1+o(1)$ and then by Chebyshev's inequality~\eqref{eq:chebyshev} the statement of Theorem~~\ref{theorem_main}  follows for $\Hnmp$. 

\section{Applications of Theorem \ref{theorem_main}}\label{sec_applications}

First we obtain the following two corollaries. 
\begin{corollary}
\label{Cor_theom} Let  $r$, $\Delta\ge 2$ be integers and $H=H^{(i)}$ a sequence of $r$-uniform hypergraphs with  $n=n(i)$ vertices, $\Delta(H) \le \Delta$, $e(H) > n/r$ and $\gamma(H) = e(H) / (n-2)$. Then for $p = \omega \left( n^{-1/\gamma(H)}\right)$ the random graph $\Hnp$ contains  a copy of $H$ a.a.s., while for every $\eps>0$ we have for $p \le (1-\eps)(e/n)^{1/\gamma}$ that  $\PP(H\subseteq \Hnp)\rightarrow 0$.\end{corollary}

\begin{proof}

Since $\Delta$ is fixed and $\gamma(H)\le (1+o(1))\Delta$, the conditions ~\eqref{eq_theocond} and~\eqref{eq_npinfty} are satisfied. Theorem~\ref{theorem_main} yields the first part of the claim.
 
 Let $X$ be the number of copies of $H$ in $\Hnp$ and we estimate its expectation $\EE (X)$ as follows:
\[
\EE(X)\le n! p^{e(H)} \le 3 \sqrt{n} (1 - \eps)^{e(H)} (n/e)^2 = o(1).
\]
Now Markov's inequality $\PP(X\ge 1)\le \EE(X)$ yields the second part of the corollary.
\end{proof}

We call a hypergraph $H$ $d$-regular if every vertex of $H$ has degree $d$.
\begin{corollary}\label{Cor_theom2}
Let  $r\ge 2$ be an integer and 
 $H=H^{(i)}$ be a sequence of $\Delta$-regular $r$-uniform hypergraphs 
 where   $\Delta=\omega(\ln(n)^{1-1/r})$ but $\Delta = o(n^{1/4})$. 
 Then for every $\eps>0$  we have that $\Hnp$ contains a.a.s.\ $H$ if  $p=(1+\eps) n^{-r/\Delta}$.
 Furthermore  $\PP(H\subseteq \Hnp)\to 0$ for $p\le n^{-r/\Delta}$, i.e.\ 
 $p=n^{-r/\Delta}$ is a sharp threshold for the appearance of copies of $H$ in $\Hnp$.
\end{corollary}

\begin{proof}
Let $X$ count the copies of $H$ in $\Hnp$ and for  $p\le n^{-r/\Delta}$ we have
\[
\PP(X\ge 1)\le \EE(X)\le n! n^{-re(H)/\Delta}=n! n^{-n}=o(1).
\]

Next we bound $\gamma(H)$ as follows: 
$\Delta/r\le \gamma(H)\le \frac{\Delta}{r} \frac{(\Delta^{1/(r-1)} + 1)}{(\Delta^{1/(r-1)}-1)}$. This is obtained 
from the estimate $e_H(v) \le \min \{ \Delta v/r , {v \choose r}\}$ by considering two cases whether $v \le \Delta^{1/(r-1)} +1$ or not. Let $\eps\in(0,1)$ and notice that~\eqref{eq_theocond} is satisfied.
 It also holds that
 \begin{multline*}
 n \left((1+\eps) n^{-r/\Delta}\right)^{\gamma(H)} \Delta^{-4}\ge 
 \left((1+\eps)n^{1/\gamma(H)-r/\Delta} \Delta^{-4r(1+o(1))/\Delta}\right)^{\gamma(H)}\ge\\
 \left((1+\eps)n^{-2r/(\Delta^{1+1/(r-1)})} (1+o(1))\right)^{\gamma(H)} 
  \rightarrow \infty,
 \end{multline*}
 and  therefore Theorem~\ref{theorem_main} is applicable and the statement follows. 
\end{proof}

Thus Theorem~\ref{theorem_main} (Corollaries~\ref{Cor_theom} and~\ref{Cor_theom2}) states that under some technical conditions the threshold for the appearance of the spanning structure comes from the expectation threshold defined in the introduction.  Further it should be noted that the appearance of $1$-overlapping Hamilton cycles and also perfect matchings and of general $F$-factors the structure in question appears as soon as some local obstruction (isolated vertices, no vertices in some copy of a fixed graph $F$) disappears. 
Thus, there seem to be two types of behaviour that are responsible for the threshold for the appearance of a bounded degree spanning structure.

In the following we derive asymptotically optimal thresholds for the appearance of various spanning structures in $\Hnp$
which are consequences of the Corollaries~\ref{Cor_theom} and~\ref{Cor_theom2}.
\subsection{Hamilton Cycles} The following is a slightly weaker version of Dudek and Frieze~\cite{Dudek-Frieze_hamiltonian}.
\begin{corollary}
 For all integers $r>\ell\ge 2$, $(r-\ell) | n$ and $p=\omega(n^{\ell-r})$ the random hypergraph $\mathcal{H}^{(r)}(n,p)$ is  $\ell$-hamiltonian a.a.s. 
\end{corollary}

\begin{proof}
Denote by $\Crl$ an $\ell$-overlapping Hamilton cycle on $n$ vertices. It is not difficult to see that $\gamma(\Crl)=\frac{n}{(r-\ell)(n-2)}$. Indeed, let $V \subseteq [n]$ be a set of size $v<n$. Then $\Crl[V]$ is a union of vertex-disjoint $\ell$-overlapping paths, where an $\ell$-overlapping path of length $s$ consists of $s(r-\ell)+\ell$ ordered vertices and edges are consecutive segments intersecting in $\ell$ vertices. This gives: $e(\Crl[V])\le (v-\ell)/(r-\ell)$ and from 
$\frac{v-\ell}{(r-\ell)(v-2)}\le \frac{n}{(r-\ell)(n-2)}$ we get $\gamma(\Crl)=\frac{n}{(r-\ell)(n-2)}$. 

Since $e(\Crl)>n/r$, $\Delta(\Crl)=\lceil\frac{r}{r-\ell}\rceil$ and $n^{2(r-\ell)/n}\rightarrow 1$, Corollary~\ref{Cor_theom} implies the statement.
 \end{proof}

\subsection{Cube-hypergraphs}
The $r$-uniform $d$-dimensional cube-hypergraph $Q^{(r)}(d)$ was studied in~\cite{Burosch-Ceccherini_cube-hypergraphs} and its vertex set is $V:=[r]^d$ and its hyperedges are $r$-sets of the vertex set $V$ that all differ in one coordinate. Thus, $Q^{(r)}(d)$ has $r^d$  vertices, $dr^{d-1}$ edges and is $d$-regular. In the case $r=2$ this is the usual definition of the (graph) hypercube. The following corollary is a direct consequence of Corollary~\ref{Cor_theom2}. 

\begin{corollary}\label{hypercubes}
For all integers $r\ge 2$, $\eps>0$ and $p=r^{-r}+\eps$ it holds $\PP(\Qrd\subseteq \cH^{({r})}(r^d,p))$ tends to $1$ as $d$ tends to infinity. On the other hand, $\PP(\Qrd\subseteq \cH^{({r})}(r^d,r^{-r}))\rightarrow 0$ as $d\rightarrow\infty$.
\end{corollary}
We remark that, in the case $r=2$, Riordan~\cite{Riordan} proved even better dependence of $\eps$ on $d$, and similar dependence can be shown for $r>2$.
\subsection{Lattices}
Another example considered in~\cite{Riordan} was the graph of the lattice $L_k$, whose vertex set is $[k]^2$ and two vertices are adjacent if their Euclidean distance is one. There it is shown that $p=n^{-1/2}$ is asymptotically the threshold. One can view $L_k$ as the cubes $Q^{(2)}(2)$ (these are cycles $C_4$) glued `along' the edges.  We define the $\ell$-overlapping hyperlattice $\Lrlk$ as the $r$-uniform hypergraph where we glue together $(k-1)^2$ copies of $Q^{({r})}(2)$ that overlap on $\ell$ hyperedges accordingly. Thus, $L^{(2)}(1,k)$ is just the usual graph lattice $L_k$. 
\begin{corollary} Let $r\ge 2$ and $k$ be an integer.
For $p = \omega\left( n^{-1/2}  \right)$ (where $n=(k-2+r)^2$) the random $r$-uniform hypergraph $\Hnp$ contains  a copy of $L^{(r)}(r-1,k)$ a.a.s. Moreover, for $p=n^{-1/2}$, $\PP(L^{({r})}(r-1,k)\subseteq \Hnp)\rightarrow 0$ as $k$ (and thus $n$) tends to infinity.  
\end{corollary}
\begin{proof}
 Observe that $L:=L^{(r)}(r-1,k)$ has $(k-2+r)^2$ vertices (which can be associated with $[k-2+r]^2$) and $2(k-1)(k-2+r)$ edges. 
 
 We aim to show that $e_L(v) \le 2 (v- r)$ for all $v \ge r+1$. We argue similarly as in~\cite{Riordan}. Observe that $e_L(v) \le 2$ for $v=r+1$. Let now $L'$ be an arbitrary subhypergraph of $L$ on $v+1\le (k-2+r)^2$ vertices such that $e(L')=e_L(v+1)$. It is easy to see that there is a vertex of degree $2$ in $L'$ (take $(i,j)$ such that $(i+1,j)$, $(i,j+1)\not\in V(L')$). It follows that then $e_L(v+1) \le e_L(v) + 2$ for $v > r+1$ giving  $e_L(v) \le 2 (v- r)$ for all $v \ge r+1$.

It follows that  $\gamma(L) \le 2$ and applying Theorem~\ref{theorem_main} with $n p^{\gamma} = \omega(1)$ yields the first part. Markov's inequality yields the second part.  
\end{proof}

\subsection{Spheres}
Let $r\ge 3$ and let $G$ be a planar graph on $n$ vertices with a drawing all of whose faces are cycles of length $r$. We define 
a sphere $S^r_n$ as an $r$-uniform hypergraph all of whose edges correspond to the faces of that particular drawing (note that a sphere is not unique). 
Observe that we get from Euler's formula for planar graphs the condition $2 v(\Srn) - 4=(r-2) e(\Srn)$. 

\begin{corollary}
Let $r \ge 3$ and $S$ be some sphere $\Srn$ with $\Delta=\Delta(\Srn)$. Then for $p=\omega \left( \Delta^{2r-4}  n^{-(r-2)/2} \right)$ the random $r$-uniform hypergraph $\Hnp$ contains a copy of $S$ a.a.s.
\end{corollary}

\begin{proof} 
From Euler's formula it follows that $e_S(v)\le \frac{2v-4}{r-2}$ and therefore $\gamma(S)=2/(r-2)$. 
Since this is an upper bound for the number of $r$-edges in this induced hypergraph we immediately get $\gamma=2/(r-2)$. The statement follows now directly from Theorem~\ref{theorem_main}. 
\end{proof}

\subsection{Powers of $(r-1)$-overlapping Hamilton cycles}
 Consider an  $(r-1)$-overlapping Hamilton cycle $C^{(r,r-1)}$ with $n$ vertices which are ordered cyclically. Given an integer $i$, we define an $i$-th power $C^{(r)}(i)$ of $C^{(r,r-1)}$ to consist of all $r$-tuples $e$ such that the maximum distance in this cyclic ordering between any two vertices in $e$ is at most $r+i-2$. In the graph case, the threshold for the appearance of $C^{(2)}(i)$ follows from Riordan's result~\cite{Riordan} for $i\ge 3$ (see~\cite{KO12}) and in the case $i=2$  an approximate threshold  due to K\"uhn and Osthus~\cite{KO12} is known.  If we count the edges of $C^{(r)}(i)$ by their leftmost vertex we get $e(C^{(r)}(i))=n \binom{r+i-2}{r-1}$. 

\begin{theorem} Let $r\ge 3$ and $i\ge 2$ be integers. 
Suppose that $p = \omega (n^{-1/\binom{r+i-2}{r-1}})$, then the random hypergraph $\Hnp$ contains a.a.s. a copy of $C^{(r)}(i)$. This threshold is asymptotically optimal. 
\end{theorem}

\begin{proof}
One can argue similarly to Proposition~8.2 in~\cite{KO12} to show $\gamma(C^{(r)}(i))\le \binom{r+i-2}{r-1}+O_{r,i}(1/n)$. The statement follows from Theorem~\ref{theorem_main}. We omit the details.
\end{proof}

\section{Proof of Theorem~\ref{theom_universality}}\label{sec:proof_universality}

\subsection{Outline}
Our proof follows a similar strategy as the one of Dellamonica, Kohayakawa, R{\"o}dl and Ruci{\'n}ski~\cite{DKRR15} 
for universality in random graphs, but to verify good properties of a random hypergraph necessary for embeddings we combine it with the recent approach of Ferber, Nenadov and Peter~\cite{ferber2013universality} who studied random graphs as well.

We will embed any bounded degree hypergraph $F\in\FnD$ into the random hypergraph $\cH=\cH^{({r})}\left(n,p\right)$ with $p=C(\ln n/n)^{1/\Delta}$ in stages. For this we will partition most of the vertices of $F$ into $3$-independent sets $X_1$,\ldots, $X_t$  (this is achieved by coloring greedily the third power of the shadow graph of $F$) and the remaining vertices will form the set of linear size $N_F(X_t)$, where the hypergraphs $F[N_F(x)]$ and the link of $x$ in $F$ look the same for all $x\in X_t$. The random hypergraph $\cH$ is then prepared as follows: the vertex set of $\cH$ is partitioned into sets $V_0$, $V_1$,\ldots, $V_t$ where ``most'' of the vertices lie in $V_0$. The property that we use first is that one can embed into $\cH[V_0]$ the induced hypergraph $F[N_F(X_t)]$ for \emph{any} $F\in \FnD$ so that the restrictions on future images for $\cup_{i\in[t]}X_i$ still offer many choices. In later rounds, we embed each $X_i$ into available vertices from $V_0\cup \bigcup_{j\le i} V_j$ by Hall's condition ($X_i$s are $3$-independent). In order to verify Hall's condition for small subsets of $X_i$ the sets $V_i$ will assist us in this.

\subsection{Proof of Theorem~\ref{theom_universality}}

Let $H=(V,E)$ be an $r$-uniform hypergraph.
The link of $v$ in $H$ is a subset of $V \choose r-1$ consisting of all $(r-1)$-sets of vertices which form an edge together with $v$
\begin{align*}
 \link_H(v) = \left\{ e^\prime \in \binom{V}{r-1} : e^\prime \cup \{ v \} \in E \right\}.
\end{align*}
For a hypergraph $H$ and a vertex $v$ we define its \emph{profile} $P_H(v)$ in $H$ as follows
\[
P_H(v)=(N_H(v),E(H[N_H(v)]),\link_H(v)) 
\]
and say that two profiles $P_H(v_1)$ and $P_H(v_2)$ are equivalent ($P_H(v_1) \cong P_H(v_2)$) if there is an isomorphism $\varphi$  that takes $H[N_H(v_1)]$ to $H[N_H(v_2)]$ and $(N_H(v_1),\link_H(v_1))$ to $(N_H(v_1),\link_H(v_2))$.
We call $N_H(v)$ the vertices of the profile.

Let $\PnD$ be the set of all possible profiles $(Z,E_1,E_2)$ that we encounter for any $F \in \cup_{n\in\Nat}\FnD$ (up to equivalence). 
Then any $|Z|\le (r-1)\Delta$, $(Z,E_1)$ is an $r$-uniform hypergraph with maximum degree $\Delta-1$ and $(Z,E_2)$ is an $(r-1)$-uniform hypergraph with at most $\Delta$ edges and without isolated vertices.
It is not difficult to bound $|\PnD|$ by a function exponential is some polynomial  in $\Delta$, but since $\Delta$ is a constant, all we will care about is that $|\PnD|$ is a constant as well that depends on $\Delta$ only.

The following lemma prepares any $F\in\FnD$ for future embedding into $\Hnp$.

\begin{lemma}
\label{lem_partition}
 Let $r \ge 2$ and $\Delta \ge 1$ be integers. Then for $t=r^3\Delta^3$, any $\eps \le |\PnD|^{-1}(t-1)^{-1}$ and any $F \in \FnD$ there exists a partition of $V(F)$ in $X_0 \cup \dots \cup X_t$ (where some $X_i$s might be empty) with the following conditions:
\begin{enumerate}
 \item $|X_t|=\eps n$, $X_0 = N_H(X_t)$,\label{prop:X0}
 \item every $x \in X_t$ has the same profile $P_F(x)$ (up to equivalence), \label{prop:profiles}
 \item and $X_i$ is $3$-independent for $i=1,\dots,t$.\label{prop:indep}
\end{enumerate}
\end{lemma}
\begin{proof}
 Let $F\in\FnD$ be given and $G$ be its shadow graph.  
  The third power $G^3$ of $G$ is the graph which we obtain by connecting any pair of vertices of distance at most $3$ by an edge. We estimate the maximum degree of $G^3$ as follows: $\Delta(G^3)\le (r-1)^3\Delta^3$. Clearly, $G^3$ is $(t-1)$-colorable and let $Y_1$,\ldots, $Y_{t-1}$ be the color sets in some coloring of $V(G^3)$ such that $|Y_1|\le |Y_2|\le \ldots \le |Y_{t-1}|$.  The sets $Y_i$ are $3$-independent in $F$ as well because the shadow of a path of length $3$ in $F$ contains a path of length $3$ in $G$, which gives an edge in $G^3$ in contradiction to the  coloring above.

 We can choose a subset $X_{t} \subseteq Y_{t-1}$ of size $\eps n \le n|\PnD|^{-1}(t-1)^{-1}$ of vertices with the same profile in $F$ (up to equivalence). 
 We set $X_0:=N_F(X_t)$ and define $X_i=Y_i \setminus X_0$ for $i=1,\dots,t-2$ and $X_{t-1}=Y_{t-1} \setminus (X_0\cup X_{t})$. The partition $V(F)=X_0 \cup \dots \cup X_t$ satisfies the required conditions.

\end{proof}
  Given a partition of $V(F)$ from the above lemma, it follows from properties~\eqref{prop:X0}--\eqref{prop:indep} that $F[X_0]$ is the disjoint union of $\eps n$ copies of the same $r$-uniform hypergraph isomorphic to $F[N_F(x)]$ for all $x \in X_t$. Furthermore, the third condition implies that any edge $e \in E(H)$ intersects each $X_i$ in at most one vertex for $i=1$,\ldots, $t$.

Let $H=(V,E)$ be an $r$-uniform hypergraph. 
Let $\mathcal{L}$ be a family of pairwise disjoint $k$-subsets of $V \choose r-1$ and we write $V(\mathcal{L})$ for $\cup_{e\in L: L \in \mathcal{L}} e$. For a subset $W \subseteq V \setminus V(\cL)$ we define the auxiliary bipartite  graph $B(H,\mathcal{L},W)$  with the vertex classes $\cL$ and $W$, where $L \in \cL$ and $w \in W$ form an edge if and only if $L \subseteq \link_H(w)$. Roughly speaking, for every unembedded $x\in V(F)$ the set $L=L_x\in\cL$ will consist of  the images of the already fully embedded $(r-1)$-sets from the  $\link_F(x)$ and the following definition which resembles the one of good graphs from~\cite{ferber2013universality} provides  essential properties that will assist us while embedding $F$ into $\Hnp$.

\begin{definition}
\label{def_good}
 We say that an $r$-uniform hypergraph $H$ is $(n,r,p,t,\eps,\Delta)$-good if there exists a partition $V(H)=V_0 \cup V_1 \cup \dots \cup V_t$, where $|V_i|=\eps n/(10t)$ for $i=1,\dots,t$, and $|V_0|=(1-\eps/10)n$ that satisfies the following conditions:
\begin{enumerate}
\item \label{prop_1}
For any profile $(Z,E_1,E_2) \in \PnD$ there exists a family $\mathcal{F}$ of $\eps n$ vertex-disjoint copies of the profile $(Z,E_1,E_2)$ with vertices in $V_0$ and edges $E_1$ present in $H$.
This family induces a family $\mathcal{F}_2$ of pairwise disjoint copies of $E_2$ in $V_0 \choose r-1$.
Furthermore, for any $W \subseteq V(H) \setminus V(\mathcal{F}_2)$ with $|W| \le (p/2)^{-\Delta}/2$ 
\begin{align*}
 |N_{B(H,\mathcal{F}_2,W)}(W)| \ge (p/2)^\Delta |W| \eps n/4
\end{align*}
holds.

 \item \label{prop_2}
 Let $1 \le k \le \Delta$ and $\mathcal{L}$ be any collection of disjoint $k$-subsets of ${V(H) \choose r-1}$. 
If $|\mathcal{L}| \le (p/2)^{-k} / 2$, then for any $i=1,\dots,t$ with $V(\mathcal{L}) \cap V_i = \emptyset$ we have
\begin{align*}
 |N_{B(H,\mathcal{L},V_i)}(\mathcal{L})| \ge (p/2)^k |\mathcal{L}| \, |V_i|/4.
\end{align*}

\item \label{prop_3}
Let $1 \le k \le \Delta$ and $\mathcal{L}$ be any collection of disjoint $k$-subsets of ${V(H) \choose r-1}$. 
If $|\mathcal{L}| \ge C^\prime (p/2)^{-k} \ln n$, then for any $W \subseteq V(H) \setminus V(\mathcal{L})$ with $|W|\ge C^\prime (p/2)^{- k} \ln n$ the graph $B(H,\mathcal{L},W)$ has at least one edge, where $C^\prime=k(r-1)+2$.
\end{enumerate}
\end{definition}

The following two lemmas establish the connection between $\Hnp$, good hypergraphs and $\FnD$-universality.

\begin{lemma}
\label{lem_Hnpgood}
 For integers $r\ge 2$, $\Delta \ge 1$, $t \ge 1$  and $\eps \le 1/(r \Delta)$, there exists a $C>0$ such that for $p \ge C \left( \ln n / n \right)^{1/\Delta}$ the random  hypergraph $\Hnp$ is $\left( n,r,p,t,\eps,\Delta \right)$-good a.a.s.
\end{lemma}

\begin{lemma}
\label{lem_gooduniversal}
 For integers $r\ge 2$, $\Delta \ge 1$  and $\eps \le |\PnD|^{-1} r^{-3}\Delta^{-3}$, there exists a $C>0$ such that  for $p\ge C \left( \ln n / n \right)^{1/\Delta}$, every $\left( n,r,p,r^3\Delta^3,\eps,\Delta \right)$-good hypergraph is $\FnD$-universal.
\end{lemma}

The proof of Theorem~\ref{theom_universality} follows immediately from Lemmas~\ref{lem_Hnpgood} and~\ref{lem_gooduniversal}. Thus, it remains to prove both lemmas.

We will make use of the following version of Chernoff's inequality, see e.g.~\cite[Theorem~2.8]{JansonLuczakRucinski_RandomGraphs}.
\begin{theorem}[Chernoff's inequality]
 Let $X$ be the sum of independent binomial random variables, then for any $\delta \ge 0$
\begin{align*}
 \PP [X \le (1 - \delta) \EE(X)] \le \exp (-\delta^2 \EE (X)/2).
\end{align*}
\end{theorem}

For an $r$-uniform hypergraph $F$ with $v(F)\ge r$, we define $d^{(1)}(F):=\frac{e(F)}{v(F)-1}$ and we also set $d^{(1)}(F)=0$ if $v(F)\le r-1$. Now we set $m^{(1)}(F):=\max\left\{d^{(1)}(F')\colon F'\subseteq F\right\}$.
We will use the following theorem that deals with almost spanning factors in random (hyper-)graphs. 
\begin{theorem}\label{thm:almost_factor}[Theorem~4.9 in~\cite{JansonLuczakRucinski_RandomGraphs}]
For every $r$-uniform hypergraph $F$ and every $\eps< 1/v(F)$ there is a $C>0$ such that for $p\ge C n^{-1/m^{(1)}(F)}$, the random hypergraph $\Hnp$ contains $\eps n$ vertex-disjoint copies of $F$ a.a.s.
\end{theorem}

\begin{proof}[Proof of Lemma~\ref{lem_Hnpgood}]
 Let $r$, $\Delta$, $t$ and $\eps\le 1/(r \Delta)$ be given, furthermore we assume that $p \ge C \left( \ln n / n \right)^{1/\Delta}$, where $C$ is a sufficiently large constant that depends only on $\eps$, $r$, $\Delta$ and $t$. We will not specify $C$ explicitly but it will be clear from the context how it should be chosen. 
 
We expose $\Hnp$ in two rounds and  write $\Hnp = \cH^{(r)}(n,p_1) \cup \cH^{(r)}(n,p_2)$, where $p_1=p_2 \ge p/2$ sucht that $(1-p)=(1-p_1)(1-p_2)$. 
In the first round we will find the familes $\mathcal{F}$ and in the second round we show that properties~\eqref{prop_1}-\eqref{prop_3} of  Definition~\ref{def_good} all hold with probability $1-o(1)$. In the beginning we  arbitrarily partition $V$ into $V_0 \cup V_1 \cup \dots \cup V_t$ such that  $|V_0|=n-\eps n/10$ and $V_i=\eps n/(10t)$ for $i=1,\dots,t$. 

\textbf{1\textsuperscript{st} round.} For a given profile  $(Z,E_1,E_2) \in \PnD$ we have that the maximum degree of  $G=(Z,E_1)$ is at most $\Delta-1$. We estimate $m^{(1)}(G)\le \max_{s\ge r}\frac{(\Delta-1) s}{r (s-1)}\le \Delta-1$. 
Theorem~\ref{thm:almost_factor} implies that there exist $\eps n$ vertex-disjoint copies of $G$ in $\cH^{({r})}(n,p_1)$ all of whose vertices are contained inside $V_0$ a.a.s. Indeed, we apply Theorem~\ref{thm:almost_factor} to $\cH^{({r})}(|V_0|,p)$ where $|V_0|\ge (1-\eps/10)n > (r-1)\Delta \eps n +\eps n/10$. We denote this family by $\cF_{(Z,E_1)}$. 

Since there are constantly many (at most $|\PnD|$) $r$-uniform hypergraphs $G$ on at most $(r-1)\Delta$ vertices with maximum degree $\Delta-1$, we will find simultaneously $\eps n$ vertex-disjoint copies of any such $G$ a.a.s.\ within $V_0$.  
Therefore, with a given profile $(Z,E_1,E_2) \in \PnD$, we associate a family $\cF$  of $\eps n$ 
vertex-disjoint copies $(Y,E',E'')$ with $(Y,E')\in \cF_{(Z,E_1)}$ and such that $(Y,E',E'')\cong (Z,E_1,E_2)$. This gives us a family $\cF_2$ of the $E''$s for such a profile, thus showing the first part of the property~\eqref{prop_1} of Definition~\ref{def_good}.

\textbf{2\textsuperscript{nd} round.}
From now on we work in $\cH=\cH^{({r})}(n,p_2)$. 

Fix some profile $(Z,E_1,E_2) \in \PnD$ and the corresponding family $\cF$ found in the first round. 
The family $\F$ induces a family $\mathcal{F}_2$ of disjoint copies of $E_2$ in $V_0 \choose r-1$.
Let $W$ be a subset of $V(\cH) \setminus V(\mathcal{F}_2)$ with $|W| \le (p/2)^{- \Delta}/2$.
For every $L \in \F_2$ let $X_L$ be the random variable with $X_L=1$ if and only if $L \subseteq \link_\cH(w)$ for some $w \in W$.
This gives us $|N_{B(\cH,\mathcal{F}_2,W)}(W)| = \sum_{L \in \mathcal{F}_2} X_L$.
The $X_L$ are independent and since $\PP[ xL\in E(B(\cH,\mathcal{F}_2,W))]\ge (p/2)^\Delta$, we compute
\begin{align*}
 \PP [X_L=0] \le (1 - (p/2)^\Delta)^{|W|} \le 1 - |W| (p/2)^\Delta + |W|^2 (p/2)^{2\Delta} \le 1 - |W| (p/2)^\Delta /2.
\end{align*}
From this we obtain
\begin{align*}
 \EE \left[ \sum_{L \in \mathcal{F}_2} X_L \right] \ge (p/2)^\Delta |W| |\mathcal{F}_2|/2\overset{|\cF_2|=\eps n}{\ge} \eps(C/2)^\Delta  |W|(\ln n) /2 
\end{align*}
and using Chernoff's inequality with $\delta=1/2$ we get
\begin{align}\label{eq:prop_1_W}
 \PP \left[ \sum_{L \in \mathcal{F}_2} X_L < (p/2)^\Delta  |W| \, |\mathcal{F}_2|/4 \right] \le \exp(-\eps(C/2)^\Delta  |W|(\ln n) /16) = n^{-\eps (C/2)^\Delta  |W| /16}.
\end{align}
Since there are at most $n^s$ choices for a set $W$ of size $s$ we can bound, for $C$ large enough, the probability that there is a set $W$ violating property~\eqref{prop_1} for $\mathcal{F}_2$ by $o(1)$.

The number of different profiles in  $\PnD$ depends only on $\Delta$ and thus also the number of $\cF_2$s. 
Thus taking the union bound over the probability that there is a set $W$ violating our condition for some family $\cF_2$ is still $o(1)$. This verifies property~\eqref{prop_1} of Definition~\ref{def_good}.

To verify properties~\eqref{prop_2} and~\eqref{prop_3} of Definition~\ref{def_good}, we use the edges of $\cH^{({r})}(n,p_2)$. Let $k \in [\Delta]$, $\mathcal{L}$ be a collection of disjoint $k$-subsets of $V \choose r-1$ with $|\mathcal{L}| \le (p/2)^{-k}/2$ and $i \in [t]$ such that $V(\mathcal{L}) \cap V_i = \emptyset$.
For $v \in V_i$, let $X_v$ be the random variable with $X_v=1$ if and only if $L \subseteq \link_\cH(v)$ for some $L \in \mathcal{L}$.
Thus $|N_{B(\cH,\cL,V_i))}(\cL)| = \sum_{v \in V_i} X_v$.
As above we obtain
\begin{align*}
\PP[X_v=0] \le \left(1-(p/2)^k\right)^{|\mathcal{L}|} \le 1 - |\mathcal{L}| (p/2)^k + |\mathcal{L}|^2 (p/2)^{2k} \le 1-|\mathcal{L}|(p/2)^k /2.
\end{align*}
We have
\begin{align*}
\EE \left[ \sum_{v \in V_i} X_v \right] \ge (p/2)^k |\mathcal{L}| |V_i| /2 \overset{|V_i|=\frac{\eps n}{10 t}}{\ge} \eps (C/2)^k  |\cL| (\ln n) /(20t)
\end{align*}
and from Chernoff's inequality with $\delta=1/2$ we get
\begin{align*}
\PP\left[ \sum_{v \in V_i} X_v \le (p/2)^k |\mathcal{L}| |V_i| /2 \right] \le 
\exp(- (p/2)^k |\mathcal{L}| |V_i| /16)\le 
n^{-\eps(C/2)^k |\cL|/(320t) }.
\end{align*}
There are less than $n^{(r-1)k|\cL|}$ possibilities to choose $\cL$. Therefore, for $C$ large enough, the probability that there exists $k \in [\Delta]$ and sets $\mathcal{L}$ and $V_i$ that violate property~\eqref{prop_2} of Definition~\ref{def_good}  is $o(1)$.

Finally, we verify that property~\eqref{prop_3} holds a.a.s.\ in $\cH$. For this we set $\ell =C' (p/2)^{-k}  \ln n$ and let $k \in [\Delta]$. 
It suffices to consider only  sets $\cL$  and $W \subseteq V \setminus V(\cL)$ each of size $\ell$.
For two such sets $\mathcal{L}$ and $W$ the probability that an edge in $B(\cH,\mathcal{L},W)$ is present equals $(p/2)^k$ and therefore the probability that there are no edges is $(1-(p/2)^k)^{\ell^2} \le \exp(-\ell^2 (p/2)^k)$.

There are less than $n^{(r-1)k\ell}$ choices for $\cL$ and less than $n^\ell$ choices for $W$.
Thus we can bound the probability that there are sets $\mathcal{L}$ and $W$ of size $\ell$ violating property~\eqref{prop_3} by
\begin{align*}
\exp [ ((r-1)k\ell+\ell)\ln n - \ell^2 (p/2)^k] = \exp[  ((r-1)k + 1 - C^\prime)\ell\ln n  ]  =o(1).
\end{align*}

\end{proof}

\begin{proof}[Proof of Lemma~\ref{lem_gooduniversal}]
Let $r$, $\Delta$,  $\eps\le |\PnD|^{-1}r^{-3} \Delta^{-3}$ be given and let $C_{\ref{lem_Hnpgood}}$ be a constant as asserted by Lemma~\ref{lem_Hnpgood} on input $r$, $\Delta$, $t:=r^3\Delta^3$ and $\eps$.  Furthermore we assume that $p \ge C \left( \ln n / n \right)^{1/\Delta}$, where $C$ is a sufficiently large constant that depends only on $\eps$, $r$, $\Delta$ and $C_{\ref{lem_Hnpgood}}$. We will not specify $C$ explicitly but it will be clear from the context how it should be chosen. 

Let $H$ be an $\left( n,r,p,t,\eps,\Delta \right)$-good hypergraph and fix the partition $V_0\cup V_1 \cup \dots \cup V_t$ of $V(G)$ as specified by Definition~\ref{def_good}.
Fix any $F \in \FnD$ and apply Lemma~\ref{lem_partition} with $r$, $\Delta$, $t=r^3\Delta^3$ and $\eps$ to obtain a partition of $V(F)$ in $X_0\cup \dots \cup X_t$ with the properties~\eqref{prop:X0}--\eqref{prop:indep} from Lemma~\ref{lem_partition}.

An embedding of $F$ into $G$ is an injective map $\phi \colon V(F) \to V(H)$, where edges are mapped onto edges. 
We start with constructing an embedding $\phi_0$ that $X_0$ maps into $V_0\subset V(H)$. From Lemma~\ref{lem_partition}, property~\eqref{prop:profiles}, we know that every $x\in X_t$ has the same profile in $F$. Therefore, let $(Z,E_1,E_2)$ be the profile of any $x \in X_t$. 
By Definition~\ref{def_good}, property~\eqref{prop_1},  there is a family $\cF$ of copies of $(Z,E_1,E_2)$ with vertices in $V_0$. Since $F[X_0]$ is the disjoint union of $\eps n$ copies of $(Z,E_1)$ we can construct $\phi_0$ by mapping bijectively every copy $(N_F(x),F[N_F(x)],\link_F(x))$ to one member $(Y,E',E'')$ of $\cF$. 
This is for sure a valid embedding of $F[X_0]$ into $H$.

Now we construct $\phi_{i}$ from $\phi_{i-1}$ for $i=1,\dots,t$ by embedding $X_i$ such that $\phi_i\left(F[\cup_{j=0}^i X_j]\right)\subseteq H$. 
The available vertices for this step are $V_i^*=(V_0\cup \dots \cup V_i) \setminus \Img (\phi_{i-1})$.
For $x \in X_i$ we collect the images of the already fully embedded $(r-1)$-sets from the  $\link_F(x)$
in 
\[
 L(x) := \left\{ \phi_{i-1}(e) : e \in  \link_F(x) \cap {\bigcup_{j=0}^{i-1}X_j \choose r-1}  \right\}.
\]
Since $X_i$ is $3$-independent we have $L(x_1)\cap L(x_2) = \emptyset$ for $x_1,x_2 \in X_i$ and we set $\cL_i = \{ L(x) : x \in X_i \}$ is a collection of vertex-disjoint sets in $\binom{V(H)}{r-1}$.
A possible image for $x \in X_i$ is any $v \in V_i^*$, for which $L(x) \subseteq \link_H(v)$. 
 It remains to find an $\cL_i$-matching in $B_i=B(H,\mathcal{L}_i,V_i^*)$ since then we set $\phi_i(x):=v$ for every  edge $vL(x)$ in this matching and, since any edge $e\in E(F)$ intersects $X_i$ in at most one vertex, we obtain 
 $\phi_i\left(F[\cup_{j=0}^i X_j]\right)\subseteq H$. 
 
 To guarantee an $\cL_i$-matching in $B_i$ we  will verify Hall's condition.  Let $U \subseteq \mathcal{L}_i$ and one needs to show that  $|N_{B_i}(U)| \ge |U|$ holds.
We assume $\emptyset \not\in U$, because otherwise $N_{B_i}(U)=V_i^*$ and $|V_i^*|\ge |\cL_i|$.

 First we verify Hall's condition for all sets $U$ with $|U| \le |V_i^*| - \eps n/10$. 
Notice that there exists  a $k\in[\Delta]$ and a subset $U^\prime$ of size at least $|U| / \Delta$ and $|L|=k$ for every $L \in U^\prime$.  If $|U^\prime| \le (p/2)^{-k}/2$, then by property~\eqref{prop_2} of Definition~\ref{def_good}  we have for $C$ large enough
\begin{align*}
|N_{B_i}(U)| \ge |N_{B_i}(U^\prime)| \ge (p/2)^k |U^\prime| |V_i|/4 \ge \eps (C/2)^k   |U| (\ln n)/(40t\Delta) \ge |U|.
\end{align*}

If $(p/2)^{-k}/2 < |U^\prime| < C' (p/2)^{-k} \ln n$, then we take any subset $U^{\prime\prime}$ of size $(p/2)^{-k}/2$ and use again property~\eqref{prop_2} of Definition~\ref{def_good}  to obtain for $C$ large enough
\begin{align*}
|N_{B_i}(U)| \ge |N_{B_i}(U^{\prime\prime})| \ge (p/2)^k |U^{\prime\prime}| |V_i|/4 \ge \eps (C/2)^k |U|/(20C^\prime t\Delta) \ge |U|.
\end{align*}

If $|U^\prime| > C^\prime (p/2)^{-k} \ln n$, then $|U|> C^\prime (p/2)^{-k} \ln n$ and there are no edges between $U$ 
and $V^*_i\setminus N_{B_i}(U)$ in $B_i$.  Therefore, property~\eqref{prop_3} of 
Definition~\ref{def_good}  yields for $C$ large enough
\[
|V_i^* \setminus N_{B_i}(U)| <  C'(p/2)^{-k} \ln n  \le C'(C/2)^{-k} (n/\ln n)^{k/\Delta} \ln n \le \eps n/10,
\]
and thus $|N_{B_i}(U^\prime)| > |V_i^*| - \eps n/10$ which verifies Hall's condition in $B_i$ for $|U|\le |V^*_i|-\eps n$.

For $i=1,\dots,t-1$ it follows from $|\bigcup_{i=1}^t V_i|=\eps n/10$ and $|X_t|=\eps n$, that 
\[
|V_i^*| - |X_i|\ge (n-|\Img\phi_{i-1}|-\eps n/10)- (n-|\Img\phi_{i-1}|-\eps n) \ge 9/10 \eps n
\] and therefore  $|\cL_i| = |X_i| \le |V_i^*| - \eps n/10$. 

Therefore we find $\cL_i$-matchings in $B_i$ for $i\in[t-1]$ one after each other extending at each step our embedding. 

In the last step, we have $|V_t^*|=|X_t|=\eps n$ and, by the partitioning of $V(F)$ with $X_0=N_F(X_t)$ we have $\cL_t=\cF_2$, where $\cF_2$ is the family guaranteed by property~\eqref{prop_1} of Definition~\ref{def_good}. Since we already saw that $|N_{B_t}(U)|\ge |U|$ for all $U\subseteq \cL_t$ with $|U|\le |\cL_t|-\eps n/10$ in $B_t=B(H,\cL_t,V^*_t)$, it suffices to verify 
$|N_{B_t}(W)|\ge |W|$ for all $W\subseteq V_i^*$ with $|W|\le \eps n/10$. If $|W|> (p/2)^{-\Delta}/2$ then we take an  arbitrary subset $W'\subseteq W$  of size exactly $(p/2)^{-\Delta}/2$ and otherwise we set $W':=W$. By property~\eqref{prop_1} of Definition~\ref{def_good}, we have 
\[
|N_{B_t}(W')| \ge (p/2)^\Delta |W'| \eps n/4,
\]
which is at least $\eps n/8>\eps n/10\ge  |W|$ if $W'\subsetneq W$ and is at least $\eps (C/2)^\Delta (\ln n) |W'|/4>|W|$ if $W=W'$. Therefore, $N_{B_t}(U)\ge |U|$ for all $|U|\ge |\cL_t|-\eps n/10$ as well and there exists a (perfect) $\cL_t$-matching in $B_t$ that allows us to finish embedding $F$ into $H$.
\end{proof}

In the proof of Theorem~\ref{theom_universality} we only considered the case of constant $\Delta$. Similarly to the arguments in~\cite{ferber2013universality} this can be extended to the range when $\Delta$ is some function of $n$ but then this would affect the lower bound on the  probability $p$. Furthermore, the proof yields a randomized polynomial time algorithm that on input $\Hnp$ embeds a.a.s.\ any given $F\in\FnD$ into $\Hnp$. All steps of the proof can be performed in polynomial time and the only place where we need to use additional random bits is to split $\Hnp$ into $\cH^{({r})}(n,p_1)\cup \cH^{({r})}(n,p_2)$ and this can be done similarly as was done in~\cite{ABHKPonline}.

\section{Sparse universal hypergraphs}
\label{sec:construction}

First we observe that \emph{any} $\FnD$-universal $r$-uniform hypergraph must possess $\Omega (n^{r-r/\Delta})$ edges. Indeed, it follows e.g.\ from a result of Dudek, Frieze, Ruci\'{n}ski and \v{S}ileikis~\cite{dudek2013approximate} that for any $\Delta\ge 1$ and $r\ge 3$ the number of labelled $r$-uniform $\Delta$-regular hypergraphs on $n$ vertices (whenever $r| n\Delta$) is $\Theta\left(\frac{(\Delta n )!}{(\Delta n/r)! (r!)^{\Delta n/r}(\Delta!)^n}\right)$. Thus, the number of non-isomorphic $r$-uniform $\Delta$-regular hypergraphs on $n$ vertices is $\Omega\left(\frac{(\Delta n )!}{(\Delta n/r)! (r!)^{\Delta n/r}(\Delta!)^n n!}\right)$ and a similar bound holds for the cardinality of $\FnD$. On the other hand an $r$-uniform hypergraph with $m$ edges contains exactly $\binom{m}{n\Delta/r}$ hypergraphs with $n\Delta/r$ edges. Thus, it holds $\binom{m}{n\Delta/r}=\Omega\left(\frac{(\Delta n )!}{(\Delta n/r)! (r!)^{\Delta n/r}(\Delta!)^n n!}\right)$ and solving for $m$ yields $m=\Omega\left(n^{r-r/\Delta}\right)$.

 The random hypergraph $\Hnp$ with $p=C(\ln n/n)^{1/\Delta}$ is $\FnD$-universal (by Theorem~\ref{theom_universality}) and has $\Theta(n^{r-1/\Delta} (\ln n)^{1/\Delta})$ edges and thus the exponent in the density of $\Hnp$ is off by roughly a factor of $r$ from the lower bound  $\Omega\left(n^{-r/\Delta}\right)$ on the  density for any $\FnD$-universal hypergraph.  
 In the following we show how one can construct sparser $\FnD$-universal $r$-uniform hypergraphs out of the universal graphs from~\cite{alon2007sparse,alon2008optimal}.
 
 For a given graph $G$ we define the $r$-uniform hypergraph $\cH_r(G)$ on the vertex set $V(G)$ and $E(\cH_r(G)) = \{ f \in {V(G) \choose r} : \exists \, e \in E(G) \text{ with } e \subseteq f \}$. Given a hypergraph $H$ and a graph $G$, we say that $G$ hits $H$ if for all $f \in E(H)$ there is an $e \in E(G)$ with $e \subseteq f$. We define $\sigma(H)$ to be the smallest maximum degree $\Delta(G)$ over all graphs $G$ that hit $H$. 
 
 \begin{lemma}\label{lem:sigma}
 For $F \in \FnD$ we have $\sigma(F) \le \Delta$.
 \end{lemma}
 \begin{proof}
 Suppose there is an $F \in \FnD$ with $\sigma(F) > \Delta$ and  let $G$ be an edge-minimal graph that hits $H$ and with $\Delta(G)=\sigma(H)>\Delta$.  Take a vertex $v$ of degree at least $\Delta+1$ in $G$. If we delete any  edge $e\in E(G)$ that is incident to $v$ there will be a hyperedge $f\in E(F)$ with $e\subseteq f$ so that $f$ does not contain any edge from $E(G)\setminus\{e\}$, by the edge-minimality of $G$. 
 Thus, every  edge $e \in E(G)$ incident to $v$ corresponds to a different hyperedge $f$ in $E(F)$ with $e\subseteq f$.  
Since $\deg(v)\ge \Delta+1$ there must be at least $\Delta+1$ hyperedges in $F$ incident to $v$, a contradiction.
 \end{proof}
 
 \begin{lemma}
 \label{lem:Hruni}
 Let $r \ge 3$ and $\Delta \ge 1$ be integers and $G$ be a $\cF^{(2)}(n,\Delta)$-universal graph. 
 Then $\cH_r(G)$ is universal for the family of $r$-uniform hypergraphs $F$ with $\sigma(F) \le \Delta$.
 \end{lemma}
 \begin{proof}
  Let $F$ be an $r$-uniform hypergraph with $\sigma(F) \le \Delta$ and $n$ vertices.
 By definition of $\sigma(F)$ there exists a graph $G'$ with $n$ vertices that hits $F$ and   $G'\in\cF^{(2)}(n,\Delta)$, and thus $G'$ can be embedded into  $G$. Since $G'$ hits $H$, the hypergraph $\cH_r(G)$ is defined in such a way that any embedding of $G'$ into $G$ is an embedding of $F$ into $\cH_r(G)$. 
 \end{proof}

Lemmas~\ref{lem:sigma} and~\ref{lem:Hruni} together imply Proposition~\ref{prop:universality} that yields constructions of sparse $\FnD$-universal hypergraphs based on universal graphs from~\cite{alon2007sparse,alon2008optimal}.

Next we turn to yet another possibility to construct $\FnD$-universal hypergraphs out of universal graphs. 
 For a given graph $G$ we define an $r$-uniform hypergraph $\cK_r(G)$ on the vertex set $V(G)$ with hyperedges being the vertex sets of the copies of $K_r$ in $G$.
\begin{lemma}\label{lem:shadow_univ}
Let $\Delta\ge 1$ and $r\ge 3$ be integers. If $G$ is a $\cF^{(2)}(n,(r-1)\Delta)$-universal  graph then
 $\cK_r(G)$ is an $\FnD$-universal hypergraph. In particular, $\cK_r(G)$ has at most $e(G) \Delta(G)^{r-2}$ edges.
\end{lemma}
\begin{proof}
 We notice first that the shadow graph $F'$ of any $F\in\FnD$ has $n$ vertices and maximum degree at most $(r-1)\Delta$, thus $F'\in \cF^{(2)}(n,(r-1)\Delta)$. Let $\varphi$ be an embedding of $F'$ into $G$. But then, by definition of  $\cK_r(G)$, $\varphi$ is also an embedding of $F$ into $\cK_r(G)$.
 
 Furthermore, we can extend any edge of $G$ to a clique in at most $\Delta(G)^{r-2}$ many ways. Therefore, 
 $e\left(\cK_r(G)\right)\le e(G) \Delta(G)^{r-2}$.
\end{proof}

The universal graphs constructed in~\cite{alon2007sparse,alon2008optimal} have $m$ edges and maximum degree $O(m/n)$. 
This yields by Proposition~\ref{prop:shadow_univ} constructions of $\FnD$-universal hypergraphs which are just some polylog factors away from the constructions obtained by applying Proposition~\ref{prop:universality}. 
However, if we take an appropriate random graph $G=G(n,p)$ then we get a better bound than $e(G) \Delta(G)^{r-2}$ on the cliques $K_r$ which leads to even sparser universal hypergraphs for $r\ge 4$.
\begin{proof}[Proof of Proposition~\ref{prop:shadow_univ}]
 To obtain the first part of the statement, we  take  $G$ to be the random graph $G(n,p)$ with $p=C(\ln n/n)^{1/((r-1)\Delta)}$. By the result of Dellamonica, Kohayakawa, R{\"o}dl and Ruci{\'n}ski~\cite{DKRR15}, 
$G$ is $\cF^{(2)}(n,(r-1)\Delta)$-universal a.a.s. Moreover, it is well known that the threshold for the appearance of $K_r$ in the random graph is $n^{-2/(r-1)}$ and that the number of cliques $K_r$ in $G(n,p)$ for $p\gg n^{-2/(r-1)}$ is 
$\Theta(n^rp^{\binom{r}{2}})$ a.a.s.\ Thus, the number of  edges in $H_1:=\cK_r(G)$ is therefore $\Theta(n^rp^{\binom{r}{2}})=\Theta\left(n^{r-\frac{r}{2\Delta}}(\ln n)^{\frac{r}{2\Delta}}\right)$. Now Lemma~\ref{lem:shadow_univ} implies that
$H_1$ is $\FnD$-universal.

As for the second part, so let $\eps>0$ and we take  
 $G$ to be the random graph $G((1+\eps)n,p)$ with $p=\omega(n^{-1/((r-1)\Delta-1)} \ln^5 n)$. By the result of Conlon, Ferber, Nenadov and \v{S}kori\'{c}~\cite{conlon2015almost},
$G$ is $\cF^{(2)}(n,(r-1)\Delta)$-universal a.a.s. We also get that $G$ has $\Theta(n^rp^{\binom{r}{2}})$ many cliques $K_r$ a.a.s. This yields an $\FnD$-universal hypergraph 
$H_2:=\cK_r(G)$ with $(1+\eps)n$ vertices and $\omega\left(n^{r-\frac{\binom{r}{2}}{(r-1)\Delta-1}}(\ln n)^{5\binom{r}{2}}\right)$ edges.
 
\end{proof}



\section{Concluding remarks}\label{sec:conclude}
We believe that it is possible to further reduce the number of edges towards the lower bound $\Omega(n^{r-r / \Delta})$ by adjusting the constructions in~\cite{alon2007sparse,alon2008optimal,ACKRRS01}  to hypergraphs. 
In particular, in view of Lemma~\ref{lem:Hruni}  it would be interesting to know whether it is true that $\sigma(F) \le \lceil 2 \Delta/r\rceil$ for all $F \in \FnD$, which would lead to almost optimal universal $r$-uniform hypergraphs for all $r \ge 3$. For $F \in \FnD$ it is easy to obtain a hitting graph $G$ with at most $n \Delta/r$ edges and thus of average degree at most $2 \Delta/r$, but the maximum degree could be $(r-1)\Delta$ and it is not clear how to reduce this below $\Delta$. One way it could be done is by applying Lov\'{a}sz local lemma but it doesn't seem to get us near $\lceil 2 \Delta/r\rceil$. We leave it to future work.

The bound on $p=\Omega\left((\ln n/n)^{1/\Delta}\right)$ in Theorem~\ref{theom_universality} is presumably not optimal and it comes from the fact that any $\Delta$-set of $\binom{[n]}{r-1}$ is expected to lie in roughly $p^\Delta n=\Omega(\ln n)$ many edges (a.a.s.\ by Chernoff's inequality) and thus, some reasonable natural embedding should always succeed. As for the lower bound, so if $\binom{t-1}{r-1}\le \Delta$ and $t$ divides $n$, then by a theorem of  Johansson, Kahn and Vu~\cite{JohanssonKahnVu_FactorsInRandomGraphs}, the factor of  $K^{(r)}_{t}$ (which is a collection of $n/t$ vertex-disjoint copies of  $K^{(r)}_{t}$) is a member of  $\FnD$. Since the threshold probability~\cite{JohanssonKahnVu_FactorsInRandomGraphs} for the appearance of such factor is $\Theta\left((\ln n)^{1/\binom{t}{r}} n^{-(t-1)/\binom{t}{r}}\right)$ and in view of $n^{-(t-1)/\binom{t}{r}}<n^{-1/\binom{t-1}{r-1}}$, the best lower bound on $p$ we are aware of is $\Omega\left((\ln n)^{1/\binom{t}{r}} n^{-(t-1)/\binom{t}{r}}\right)$.

 As already mentioned in the introduction,   Ferber, Nenadov and Peter studied in~\cite{ferber2013universality} universality of $G(n,p)$ for the class of  graphs $F$ on $n$ vertices 
 with maximum degree at most $\Delta$ and the maximal density $d$. They showed that the random graph $G(n,p)$ with $p=\omega(\Delta^{12} n^{-1/(2d)}\ln^3 n)$ is universal for this family. 
  The recent almost spanning result of Conlon, Ferber, Nenadov and \v{S}kori\'{c}~\cite{conlon2015almost} builds on~\cite{ferber2013universality}. 
 We believe that similar improvements can be obtained for random $r$-uniform hypergraphs ($r\ge 3$).

\providecommand{\bysame}{\leavevmode\hbox to3em{\hrulefill}\thinspace}
\providecommand{\MR}{\relax\ifhmode\unskip\space\fi MR }
\providecommand{\MRhref}[2]{%
  \href{http://www.ams.org/mathscinet-getitem?mr=#1}{#2}
}
\providecommand{\href}[2]{#2}
\def\MR#1{\relax}

\section*{Appendix}
The only difference in the proofs of Lemmas below to their graph counterparts in~\cite{Riordan} is that we work with the shadow graph of a hypergraph. 
\begin{proof}[Sketch of the proof of Lemma~\ref{le_XFdelta}]
This lemma is a straightforward adaptation of Lemma~4.4 from~\cite{Riordan}. Let $Y_F(H)$ be the number of labelled copies of $F$ in $H$. One observes that then $Y_F(H)/Y_F(\Krn)=X_F(H)/X_F(\Krn)$ holds. Since $Y_F(\Krn)=n!$, one needs to estimate $Y_F(H)$. We will embed first exactly one vertex from each of the $k_r(F)$ nontrivial components. This can be done in $(n)_{k_r(F)}$  ways. Next we can embed $(r-1)$ vertices of each component by embedding one particular edge. This can be done in at most $\Delta (r-1)!$ ways into $H$. This gives at most $\left(\Delta (r-1)!\right)^{k_r(F)}$ possibilities in total. Finally, all  the remaining $r(F)-k_r(F)$ vertices from the nontrivial components can be embedded in at most $((r-1)\Delta)^{r(F)-k_r(F)}$ ways. The isolated vertices can be embedded in at most $k_1(F)!$ ways. We estimate $Y_F(H)\le (n)_{k_r(F)} \left(\Delta (r-1)!\right)^{k_r(F)}((r-1)\Delta)^{r(F)-k_r(F)}k_1(F)!$. We obtain:
\begin{multline*}
\frac{Y_F(H)}{Y_F(\Krn)}\le \frac{(n)_{k_r(F)} \left(\Delta (r-1)!\right)^{k_r(F)}((r-1)\Delta)^{r(F)-k_r(F)}k_1(F)!}{n!}\le \\
 \frac{\left(2 (r-1)! \Delta\right)^{r(F)}e^{r(F)+(r-2)k_r(F)}}{n^{r(F)+(r-2)k_r(F)}}. 
\end{multline*}
\end{proof}

\begin{proof}[Sketch of the proof of Lemma~\ref{le_TH´´}]
The proof is similar to the proof of Lemma~4.5~\cite{Riordan}. One rewrites $ T_H^{\prime \prime}$ by going over all good connected hypergraphs $F$  on $s$ vertices (then $r(F)=s-(r-1)$ and $k_r(F)=1$) and upper bounds the sum as follows:
\begin{multline*}
 T_H^{\prime \prime}\le \sum_{s=r+1}^n \frac{\left(4 (r-1)! \Delta\right)^{s-r+1} e^{s-1}}{n^{s-1}}
 \sum_{V} \sum_{m=0}^{e_H(s)} {e_H(s) \choose m} (p^{-1}-1)^m \\
\le e^{r-2}\sum_{s=r+1}^n \frac{\left(12 r! \Delta\right)^{s-r+1}}{n^{s-1}} \sum_{V} p^{-e_H(s)},
\end{multline*}
where the second sum is over all $s$-element sets $V$ such that $H[V]$ is connected.

We consider the shadow graph $H^\prime$ of $H[V]$. Now every $V \subseteq [n]$ as above also induces a subgraph in $H^\prime$ which is connected and therefore contains a spanning tree. We can estimate the number of $V$s by estimating the number of labelled trees in $H^\prime$ on $s$ vertices and then unlabelling these. 

Given a labelled tree $G$ on $s$ vertices there are at most $n (\Delta(r-1)!)(\Delta(r-1))^{s-r}$ ways of mapping it into $H^\prime$: 
$n$ accounts for the first vertex of $G$, then (in $H$) we can choose next $(r-1)$ vertices at once in $\Delta(r-1)!$ ways, and finally every remaining vertex in at most $\Delta(r-1)$ ways since $\Delta(H^\prime)\le \Delta(r-1)$. We get at most $n (\Delta(r-1)!)(\Delta(r-1))^{s-r} s^{s-2}$ trees and unlabelling gives us at most 
\[
n (\Delta(r-1)!)(\Delta(r-1))^{s-r} s^{s-2}/s!\le n (\Delta r!)^{s-r+1} e^s
\]
sets $V$. This implies
\begin{align*}
 T_H^{\prime \prime} \le n e^{2r} \sum_{s=r+1}^n \left(\frac{12 r!^2 \Delta^2}{n}\right)^{s-1}  p^{-e_H(s)}.
\end{align*}
\end{proof}

\end{document}